\documentclass[10pt]{amsart}

\usepackage{geometry,enumitem}
\usepackage{hyperref}

\usepackage{color}
\usepackage{accents}

\usepackage{amsthm, amsxtra}

\numberwithin{equation}{section}

\newtheorem{theorem}{Theorem}[section]
\newtheorem*{theorem*}{Theorem}
\newtheorem{remark}[theorem]{Remark}
\newtheorem{lemma}[theorem]{Lemma}

\newtheorem{proposition}[theorem]{Proposition}

\newtheorem{definition}[theorem]{Definition}

\newcommand{\R}{\mathbb{R}}
\newcommand{\E}{\mathbb{E}}

\author{Nassif Ghoussoub, Young-Heon Kim, and Aaron Zeff Palmer}

\address{* Nassif Ghoussoub and Aaron Zeff Palmer}
\address{Department of Mathematics\\ University of British Columbia\\ Vancouver, V6T 1Z2 Canada}
\email{nassif@math.ubc.ca,  azp@math.ubc.ca}

\address{* Young-Heon Kim}
\address{Department of Mathematics\\ University of British Columbia\\ Vancouver, V6T 1Z2 Canada \&
\newline
Center for Mathematical Challenge\\ Korea Institute for Advanced study, Seoul, Korea.}
\email{yhkim@math.ubc.ca}

\title[Monge transport Brownian martingale]{A solution to the Monge transport problem for Brownian martingales}

\thanks{The  first two 
authors are partially supported by  the 
Natural Sciences and Engineering Research Council of Canada (NSERC). 
\copyright 2019 by the author.
}

\date{\today}

\begin{document}

\begin{abstract}
	We provide a solution to the problem of optimal transport by Brownian martingales in general dimensions whenever the transport cost satisfies certain subharmonic properties in the target variable, as well as a stochastic version of the standard ``twist condition" frequently used in deterministic Monge transport theory. This setting includes in particular the case of the distance cost  $c(x,y)=|x-y|$.  We prove existence and uniqueness of the solution and characterize it as the first time  Brownian motion hits  a barrier that is determined by solutions to a corresponding dual problem.
\end{abstract}

\maketitle
\tableofcontents

\section{Introduction} \label{sec:introduction}

Given a cost function $c:\mathbb{R}^d \times \R^d \to \R$,  the Monge optimal transport problem  seeks a map $T: \R^d \to \R^d$ that pushes forward a given probability measure $\mu$ to $\nu$ (written as $T_\#\mu =\nu$) while minimizing the total transport cost:
\begin{align}\label{eqn:Monge}
	\mathcal{T}_c (\mu, \nu)=\inf_{T_\#\mu =\nu} \int_{\R^d} c\big(x, T(x)\big) \mu(dx).
\end{align}
Kantorovich \cite{kantorovich1942translocation} proposed the following linear relaxation of the Monge problem, which was eventually shown in \cite{pratelli2007equality} to have the same infimum provided $c$ is continuous and $\mu$ is non-atomic: 
\begin{align}\label{eqn:Kanto}
	\mathcal{T}_c (\mu, \nu)= \inf_{\pi \in \Gamma(\mu,\nu)}\int_{\R^d\times \R^d} c(x, y)\pi(dx,dy), 
\end{align}
where $\Gamma(\mu,\nu)$ is the set of probability measures on $\R^d \times \R^d$, whose marginals are $\mu$ and $\nu$. In turn, the last expression can also be written in terms of an optimal statistical correlation problem between probability measures given by 
\begin{align}\label{eqn:Stat}
	\mathcal{T}_c (\mu, \nu)= \inf\Big\{\mathbb{E} \big[ c(X, Y)\big]; X\sim \mu, Y\sim \nu\Big\}, 
\end{align}
where $X\sim \mu$ means that the random variable $X$ has $\mu$ as its distribution.

Monge \cite{M} originally formulated Problem \eqref{eqn:Monge} with the distance cost $c(x,y) =|x-y|$, and a solution  was provided more than 200 years later \cite{E-G} \cite{ambrosio-optimal2003} \cite{caffarelli2002constructing} \cite{Trudinger-Wang2001-Monge}.  Brenier \cite{B1} on the other hand, dealt with the important case of the squared distance $c(x, y) =|x-y|^2$, which was 
extended by  Caffarelli \cite{caffarelli96-general-cost}, Gangbo and McCann \cite{G-M}, 
 Gangbo \cite{gangbo1995habilitation}, Levin \cite{levin1999abstract} and Ma-Trudinger-Wang \cite{ma2005regularity} to more general cost functions $c$ that satisfy the so-called {\em twist condition}, namely, 
\begin{align}\label{eqn:twist}
	\hbox{the map $y \mapsto \nabla_x c(x, y)$ is one-to-one for all $x$.}
\end{align}
This condition was shown to be sufficient to ensure that the minimizer of the relaxed problem \eqref{eqn:Kanto} is indeed supported on the graph of a single-valued map $T$, hence solving the original problem \eqref{eqn:Monge}. In the statistical terms of \eqref{eqn:Stat}, the result means that the optimal random variables are perfectly dependent, i.e., $Y=T\circ X$ a.e.

We also note that optimal transport problems associated to cost functions of the form $g(|x-y|)$, where $g$ is convex can be represented by ones given by the generating function of an associated Lagrangian $L$, i.e.,  
\begin{equation}
	c(x, y)=\inf\Big\{\int_0^1  L\big(t, \gamma (t), \dot \gamma (t)\big) dt;\ \gamma (0)=x,\ \gamma (1)=y\Big\}. 
\end{equation}
In this case, optimizing maps are closely related to the corresponding Hamiltonian dynamics (e.g.,  Benamou-Brenier \cite{Be-Br}, Bernard-Buffoni \cite{B-B}).

 {\it The martingale transport problem} was motivated by questions of option-pricing in mathematical finance. See for example \cite{hobson2011skorokhod} \cite{Dolinsky2014} \cite{henry2017model}. It consists of a constrained version  of (\ref{eqn:Stat}), when the pair of random variables $(X, Y)$ forms a one-step martingale, i.e.,
\begin{align}\label{eqn:Mart}
	\mathcal{M}_c (\mu, \nu)= \inf\Big\{\mathbb{E} \big[ c(X, Y)\big]; X\sim \mu, Y\sim \nu \, \hbox{and $\mathbb{E}[Y|X]=X$}\Big\}. 
\end{align}
The martingale property on $(X, Y)$ entails that $(w (X), w (Y))$ is a real valued submartingale for any convex function $w$ on $\R^d$.  This implies a necessary condition on the probability measures $\mu$ and $\nu$ for such a martingale pairing to exist, namely, that they are in convex order, i.e.,
\begin{equation*}
	\mu \prec_C \nu \quad \hbox {which means \,\, $\int  w(x)\mu(dx) \le \int w(y) \nu(dy)$ for all convex functions $w$}. 
\end{equation*} 
A remarkable theorem of Strassen \cite{strassen1965existence} states that the convex order between $\mu$ and $\nu$ is also sufficient for the existence of a one-step martingale starting at $\mu$ and ending at $\nu$, hence providing a martingale transport for Problem \eqref{eqn:Mart}. The martingale transport problem is 
by now well understood in dimension one, including the formulation of corresponding dual principles, cases where the latter are attained, as well as questions of uniqueness and structure of the optimal martingales in the primal problem (see for example \cite{hobson2011skorokhod}, \cite{Dolinsky2014} \cite{henry2017model}, \cite{beiglboeck-juillet2016}, \cite{beiglbock2017complete}, \cite{BLO17}).  The higher dimensional case however is less understood as could be seen in \cite{ghoussoub2015structure}, where a solution for dimension $d=2$ is given. 

In this paper, we shall address {\it the Brownian martingale transport problem}, which involves a more particular class of martingales and can be formulated as 
\begin{align}\label{eqn:primal}
	\mathcal{P}_c(\mu, \nu) =  \inf_{\tau} \Big\{\mathbb{E} \big[ c(B_0, B_\tau)\big]; \ B_0 \sim \mu \quad \& \quad B_\tau \sim \nu\Big\}, 
\end{align}
where $(B_t)_t$ is Brownian motion starting with distribution $\mu$ and ending at a stopping time $\tau$ such that $B_\tau$ realizes the distribution $\nu$. Note that here again, this imposes (in dimension $d\geq 2$) an even more stringent condition on the pair $(\mu, \nu)$, namely that they are in {\it subharmonic order,} i.e.
\begin{equation*}
	\mu \prec_{SH}\nu \quad \hbox {which means \,\, $\int  w(x)\mu(dx) \le \int w(y) \nu(dy)$ for all subharmonic functions $w$}. 
\end{equation*}
This problem was originally proposed by Skorokhod in one dimension \cite{skorokhod1965studies} with $\mu=\delta_0$, where he proved the existence of a randomized stopping time $\tau$ such that $B_\tau \sim \nu$, where $(B_t)_t$ is a Brownian motion starting at $0$.  The existence of a nonrandomized stopping time in one dimension was established in \cite{dubins1968theorem} and \cite{root1969existence}.  An alternative construction of a randomized stopping was given by \cite{rost1971stopping} for more general processes.  The existence of a nonrandomized stopping time in higher dimensions was established by Baxter and Chacon \cite{baxter1974potentials} and extended in \cite{falkner1980skorohod}. A thorough review of solutions to the Skorokhod problem and its connections with optimal transport has been given in \cite{obloj2004skorokhod}.





Another formulation of this problem, similar to Kantorovich's relaxation, but which restricts the {\it transport plans} to follow Brownian paths, is the following:
\begin{align}\label{eqn:KantoS}
	\mathcal{P}_c (\mu, \nu)= \inf_{\pi \in \mathcal{BM}(\mu,\nu)}\int_{\R^d\times \R^d} c(x, y)\pi(dx,dy), 
\end{align}
where each $\pi \in \mathcal{BM}(\mu,\nu)$ is a probability measure on $\R^d \times \R^d$ with marginals $\mu$ and $\nu$, satisfying
$
	\delta_x\prec_{SH} \pi_x \ {\rm for\ }\mu{\rm-a.e.}\ x, 
$
where $\pi_x$ is the disintegration of $\pi(dx,dy)=\pi_x(dy)\mu(dx)$. Again, such transport plans $\pi$ can be seen as joint distributions of $(B_0,B_\tau)\sim \pi$, where $B_0 \sim \mu$, $B_\tau \sim \nu$ and $\tau$ is a possibly randomized stopping time for the Brownian filtration. See for example \cite{ghoussoub-kim-lim-radial2017}. Under the assumption that $\mu$ and $\nu$ have compact support, the set $\mathcal{BM}(\mu, \nu)$ is a weak$^*$-compact set of measures in the Banach space dual to the bounded continuous functions, and Problem \eqref{eqn:KantoS}  has an optimal solution $\pi^* \in \mathcal{BM}(\mu, \nu)$. 	However, similar to the issue triggered by Kantorovich's relaxation of \eqref{eqn:Monge}, such a solution may very well correspond to  a randomized stopping time. A main objective of this paper is to figure out when there exists a true and unique optimal stopping time that solves the problem. 

This problem  is equivalent to the martingale problem in the one dimensional case, a context addressed by Henry-Labordere-Touzi \cite{HLT2015} and others.  The duality theory and applications to mathematical finance were developed in \cite{hobson2015robust} and \cite{beiglboeck-juillet2016}, which also establish the existence and uniqueness of an optimal nonrandomized optimal stopping time under certain conditions.   Motivated by the methods of standard mass transports, in \cite{beiglboeck2017optimal} Beiglb\"ock, Cox, and Huesmann further developed the duality theory and a monotonicity principle for problems with stochastic cost processes of the form,
$$
	\inf_{\tau}\Big\{\mathbb{E}\big[G_\tau\big];\ B_0\sim \mu,\ B_\tau\sim \nu\Big\}.
$$
In particular, they used this theory to solve the problem in the case  $G_\tau=g(\tau)$ where $g$ is strictly convex (resp.\ concave), by showing that  the unique optimal stopping time is given by the classical Root (resp.\ Rost) embedding, which is the hitting time of a space-time barrier.  The problem of the attainment of the dual problem remained a challenge. In \cite{beiglbock2017complete}, the authors prove that a relaxed dual problem is attained when $d=1$ or costs of the form (\ref{eqn:primal}), and then the attainment of the original dual problem was established for $d=1$ in \cite{BLO17}.  By using PDE methods, dual attainment was extended for $d>1$ in \cite{GKPS} to the case where the cost is an integral along the path, that is for problems of the form 
 \begin{align}\label{eqn:Lag-inf}
	\inf_\tau \Big\{\mathbb{E}\Big[ \int_0^\tau  L(t, B_t) dt\Big];\ B_0 \sim \mu,\ B_\tau \sim \nu \Big\}.
\end{align}
For these costs, it is shown in \cite{GKPS} that the unique optimal stopping time is the hitting time of a space-time barrier when $L$ is strictly increasing (resp.\ decreasing).
	
Our goal in this paper is to solve Problem \eqref{eqn:primal} in general dimensions for costs that only depend on the initial and final states and not on the path traveled between them, and we make special note of the case where the cost is given by the Euclidean distance $c(x,y)=|x-y|$. It is important to note that --unlike standard mass transports-- Problem \eqref{eqn:primal} cannot be reduced to \eqref{eqn:Lag-inf} nor to 
\begin{align} \label{eqn:stochCost}
	\inf_{\tau, A} \Big\{\E\Big[ \int_0^\tau L(t,X_t,A_t)dt\Big];\ X_0\sim \mu,\ X_\tau\sim \nu\Big\},
\end{align}
where $X_t\in \R^n$ solves the SDE: $dX_t= f(X_t,A_t)dt+\sigma dW_t.$ We mention that because, at least in (\ref{eqn:Lag-inf}), the time monotonicity of the Lagrangian cost enables one to construct an appropriate barrier set in space-time, as was done in \cite{beiglboeck2017optimal} and \cite{GKPS}, hence establishing both the uniqueness and the hitting-time characterization of the optimal stopping. A similar result was established in \cite{ghoussoub-kim-lim-radial2017} in the case where the measures $\mu$, $\nu$ and the cost $c(x,y)$ are radially symmetric.

We also note that another advantage of the Lagrangian formulation of the cost is the associated Eulerian (mass flow) formulation to the optimal transport problem \cite{Be-Br}, \cite{B-B}, which was also extended to the free-end case in both the deterministic \cite{GKP} and stochastic cases \cite{GKPS} as long as the cost is of Lagrangian type. However, this seminal approach cannot be used for the costs we are considering in this paper,  where we prove --among other things-- the following result. 
 
\begin{theorem*} [see Theorem \ref{cor:distance-unique}]
	Let $\mu$ and $\nu$ be two probability measures that are supported on an open bounded convex set $O\subset \R^d$.  Assuming  $\mu\prec_{SH}\nu$ and that they have continuous densities with respect to Lebesgue measure, then there is a unique optimal stopping time $\tau^*$ that minimizes 
\begin{align}
	\inf_{\tau} \Big\{\mathbb{E}  |B_0 -B_\tau|; \ B_0 \sim \mu \quad \& \quad B_\tau \sim \nu\Big\}.  
\end{align}
Moreover, $\tau^*$ is only randomized at $0$, while otherwise it is given by the hitting time 
\begin{equation}
	\eta = \inf\{t>0;\ (B_0,B_t)\in R\}
\end{equation}
for some measurable $R\subset \overline{O}\times \overline{O}$.
\end{theorem*}

 We shall first establish the following weak duality formula for the primal problem \eqref{eqn:primal} under the mere assumption that $c\in C(\R^d\times \R^d)$:
\begin{equation}\label{weak.dual}
		\mathcal{P}_c (\mu, \nu)=D_c(\mu,\nu):= \sup_{\psi\in LSC(\overline{O})} \Big\{\int_O \psi(y)\nu(dy) - \int_{O} J_\psi(x,x)\mu(dx)\Big\}, 
	\end{equation}
where $ LSC(\overline{O})$ is the cone of all bounded lower semi-continuous functions on $\overline O$.  Throughout the paper we will denote $B^x_t$ as the Brownian motion beginning with $B_0^x=x$, and we will use the restriction of stopping times to this set of paths without changing notation so that for continuous functions $h\in C(\R^d\times \R^d)$ and stopping times $\tau$,
$$
	\mathbb{E}\big[ h(B_0,B_\tau)\big]=\int_{\R^d} \mathbb{E}\big[h(x,B_\tau^x)\big]\mu(dx).
$$ 
The so-called value function $J_\psi$ is defined as   
\begin{align}\label{eqn:J-psi} J_\psi (x, y)= \sup_{\tau \leq \tau_O}  \mathbb{E} \big[ \psi (B^y_\tau) - c(x, B^y_\tau)  \big], 
 \end{align}
where $\tau_O$ is the first exit time of the set $O$.  We point out two other useful characterizations of the value function $J_\psi$.  Under some regularity assumptions on $\psi$ and $c$, and for each fixed $x\in \overline{O}$, the function $y\mapsto J_\psi(x,y)$ is the unique viscosity solution to the obstacle problem for $u\in C(\overline{O})$:
\begin{align*}
\begin{cases}
			u(y)\geq { \psi}(y)-{ c}(x,y),\ {\rm for}\ y\in O,\\
			u(y)= { \psi}(y)-{ c}(x,y)\ {\rm for\ }y\in \partial O,\\
			\Delta u(y)\leq 0\ {\rm for }\ y\in O,\\
			\Delta u(y)= 0\ {\rm whenever }\ u(y)> { \psi}(y)-{ c}(x,y),
		\end{cases}
\end{align*}
as well as the unique minimizer of the variational problem
\begin{align*}
	\min_{u\in H^1(O)}\Big\{\int_O \big|\nabla u(y)\big|^2dy;\ u\geq \psi-{ c}(x,\cdot),\ u-\psi+{c}(x,\cdot)\in H_0^1(O)\Big\}.
\end{align*}

We then assume the following condition on the cost:
\begin{align}\label{eqn:subharmonic-cost}
 \hbox{ $y \mapsto c(x, y)$ is subharmonic and $D$-superharmonic, i.e., $0\leq \Delta_yc(x,y)\leq D$},
\end{align}
which will yield that the dual problem $D_c(\mu,\nu)$ can actually be restricted to a weakly compact set of functions  $\mathcal{B}_D\subset H_0^1(O)$ that consists of bounded and nonpositive $D$-superharmonic functions that vanish on $\partial O$.  These functions are lower semicontinuous and satisfy a remarkable property of continuity along Brownian paths (see Lemma \ref{lem:Sobolev_stopping_time_convergence}). These properties of the maximizing set $\mathcal{B}_D$ then guarantee that the dual problem (\ref{weak.dual}) is attained by a sufficiently regular function $\psi^*$ (see Theorem~\ref{thm:strong_dual}). 

Condition (\ref{eqn:subharmonic-cost}) on the cost can be weakened in several cases. Indeed, while the distance cost, $c(x,y)=|x-y|$ does not satisfy $D$-superharmonicity at the singular points $x=y$, we can avoid this issue  under the assumption that the supports of $\mu$ and $\nu$ are strictly disjoint. We then use a localization argument to handle the general case (without the disjointness assumption) for this important cost. Moreover, the subharmonicity of $y \to  c(x, y)$ can also be replaced by the condition that  $M=\inf_{x,y \in O}\Delta_y c(x, y)>-\infty$, since in this case one can replace $c$ with the subharmonic cost $\tilde c(x, y)=c(x, y) - h(y)$ where $h$ is any function in $C^\infty (O)$ such that $\Delta h(y) \leq M.$ On the other hand, a one-dimensional counterexample to the dual attainment problem in the case where \eqref{eqn:subharmonic-cost} is not satisfied is given in \cite{beiglbock2013model}. See \cite{beiglboeck-juillet2016} and \cite{ghoussoub2015structure} for related results. 

For the issue of uniqueness, we isolate a second condition on the cost which can be seen as a stochastic counterpart of the twist condition \eqref{eqn:twist}. It can be stated as follows: For each pair of states $(x,y)$ and any stopping time $\tau$ for Brownian motion starting at $y$, 
\begin{align}\label{eqn:ST-intro}
	\mathbb{E}\big[\nabla_xc(x,B_{\tau}^y)\big] = \nabla_xc(x,y) \quad \Longrightarrow	\quad 	\tau=0.
\end{align}
We therefore call it the {\it stochastic twist condition}. We note that a related condition on cost functions was introduced in the one dimensional case in \cite{hobson2012model} as {\it increasing (resp.\ decreasing) variation swap kernel } and in \cite{HLT2015} as {\it the martingale counterpart of the Spence-Mirrlees condition}, and it reads as
\begin{equation}
c_{yyx}(x,y)>(<)0.
\end{equation}
The stochastic twist condition enables us to prove that the optimal stopping time $\tau^*$ is unique and is characterized  as the first hitting time of a barrier set determined by the dual optimizer $\psi^*$; see Theorem~\ref{thm:hitting-unique}. The delicate proof also uses the differentiability of the value function, which we obtain from a Lipschitz assumption on $x\mapsto c(x,y)$, as well as the dynamic programming principle defining $J_\psi$.
 In particular, the distance cost $|x-y|$ satisfies the stochastic twist condition when $x\not=y$; a property that was used in \cite{ghoussoub2015structure} for   the general martingale optimal transport problem, and was realized to apply to the case of Brownian martingales by Tongseok Lim.  We extend our results to this particular and important cost function in Theorems~ \ref{thm:distance-unique} and ~\ref{cor:distance-unique}.  

\begin{remark}\label{rmk:riemann-intro2} 
The results in this paper,  including Theorems~\ref{thm:distance-unique} and ~\ref{cor:distance-unique} and the related results on the dual attainment (Theorem~\ref{thm:strong_dual}),  hold in more general settings such as for Brownian motion valued in a  geodesically convex bounded domain $O$ of a complete nonpositively curved Riemannian manifold, when the cost $c$ is given by the Riemannian distance $d(x,y)$. Here, the Laplace operator is replaced with the Laplace-Beltrami operator \cite{grigor1999analytic}. Note that the stochastic twist condition then holds  for any differentiable Riemannian distance, while nonpositive curvature implies subharmonicity (in fact, convexity) of $y\mapsto d(x, y)$.  The proofs are adaptable from the Euclidean case with additional technical complications, which we do not address further. 
\end{remark}

The paper is organized as follows: 
In Section \ref{sec:weak_dual} we prove the weak duality principle \eqref{weak.dual} and reduce the problem to optimizing over the end potential by using a dynamic programming principle. 
This gives a novel point of view for  problem \eqref{eqn:primal}. In Section \ref{sec:BD}, we introduce a remarkable set of $D$-superharmonic functions $\mathcal{B}_D$, and represent it as a weakly compact  convex subset of $H^1_0(O)$. 
In Section ~\ref{sec:dual_attainment} we prove our first main result, namely the attainment in the dual problem (Theorem \ref{thm:strong_dual}), by showing that one can restrict the maximization of \eqref{weak.dual}   to the set $\mathcal{B}_D$.   Some of the key lemmas there use condition \eqref{eqn:subharmonic-cost} in a crucial way.
This then yields a barrier set for the verification theorem (Theorem~\ref{thm:verification}).  Section \ref{sec:ST} discusses the key Stochastic Twist condition \eqref{eqn:ST-intro}  and includes a few important examples. Section \ref{sec:hitting_time} contains the proof for uniqueness and the characterization of the optimal solution to \eqref{eqn:primal} as a first hitting time (Theorem \ref{thm:hitting-unique}). The case of the distance cost is finally addressed in Section \ref{sec:distance} (Theorems \ref{thm:distance-unique}, \ref{cor:distance-unique}).   The appendix contains a couple of technical results used in Sections \ref{sec:weak_dual} and \ref{sec:BD}.   Additionally, we shall use our result on dual attainment to provide there a quick proof of a version of the monotonicity principle of Beiglb\"ock, Cox, and Huesmann \cite{beiglboeck2017optimal} that is adapted to our setting. 

\section{A dual variational principle}\label{sec:weak_dual}

Duality is a key aspect of many optimization problems. 
In our case, we shall see that the dual problem to $\mathcal{P}_c(\mu, \nu)$ arises directly as the linear maximization problem
	\begin{align}\label{eqn:dual}
		\mathcal{D}'_c(\mu, \nu) = \sup_{(\psi, J)  \in \mathcal{A}_c} \Big[  \int_{\R^d} \psi (y)\nu(dy)  - \int_{\R^d} J(x, x) \mu(dx) \Big]
	\end{align}
	with linear constraints
		\begin{align}\label{eqn:A_c}
 		\mathcal{A}_c =\big\{ (\psi, J)\in C(\R^d)\times C(\R^d\times\R^d);\  J(x,y)\geq \psi(y)-c(x,y)\, {\rm and}\, \Delta_yJ(x,y) \leq 0 \big\},
	\end{align}
	where we understand the inequality $\Delta_yJ(x,y) \leq 0$ in the sense of viscosity.  The equivalence a probabilistic notions of superharmonic functions will be recalled in Lemma \ref{lem:superharmonic}, and later with a weak and variational notions in Proposition \ref{lem:BD_properties} and Lemma \ref{lem:variational}.

	Define the operator $V_{c,x}^\psi$ as 
	\begin{align*}
 		V_{c,x}^\psi[J](y):=\min\big\{J(x, y) -\psi(y) + c(x, y), -\Delta_y J(x, y) \big\},
	\end{align*}
	we can then understand $\mathcal{A}_c$ as the set of pairs $(\psi,J)$ such that for each $x\in \R^d$, $y\mapsto J(x,y)$ is a supersolution of $V_{c,x}^\psi [J(x,\cdot)](y) \ge 0$ in the sense of viscosity. 

	We assume that the measures $\mu$ and $\nu$ have support on a bounded open convex domain $O$.  In this case we can determine $J$ in terms of $\psi$ as the minimal viscosity supersolution.  We call this the value function, and it is given by
	\begin{align}\label{eqn:J_psi}
		J_\psi(x,y):= \sup_{\tau\leq \tau_O}\mathbb{E}\big[\psi(B^y_\tau)-c(x,B^y_\tau)\big],
	\end{align}
	where $B^y_t$ is the Brownian motion with $B^y_0=y$, $\tau_O$ is the exit time from $O$, and the supremum is over all  stopping times prior to $\tau_O$.  In the following we denote by $LSC(\overline{O})$ the set of bounded, lower semicontinuous functions on $\overline{O}$. Recall the definition of Brownian martingale plans between $\mu$ and $\nu$ as the finite measures $\pi$ on $O\times O$ with marginals $\mu$ and $\nu$ that disintegrate as $\pi(dx,dy)= \pi_x(dy) \mu(dx)$ in such a way that $\delta_x\prec_{SH} \pi_x\ {\rm for}\ \mu$-a.e.  $x\in O$. We set $\mathcal{BM}(\mu,\nu)$ to be the set of such transport plans, i.e. 
	\begin{align}\label{eqn:BM}
		\mathcal{BM}(\mu,\nu) = \big\{\pi\geq 0,\  \pi(\cdot, O)=\mu,\  \pi(O,\cdot)=\nu,\ \delta_x\prec_{SH} \pi_x\ {\rm for}\ \mu\hbox{-a.e.}\ x\in O\big\}.
	\end{align}

	\begin{theorem}\label{thm:weakdual} Suppose that $c\in C(\R^d\times\R^d)$ and $\mu$ and $\nu$ are probability measures with support in an open, bounded and convex subset $O$ of $\R^d$.  Then, 
		\begin{align}\label{eqn:weak_duality}
 			\mathcal{P}_c(\mu, \nu) = \mathcal{D}'_{c}(\mu, \nu)=\mathcal{D}_{c}(\mu, \nu) := \sup_{\psi\in LSC(\overline{O})}\Big\{\int_O \psi(y)\nu(dy)-\int_O J_\psi(x,x)\mu(dx)\Big\}, 
		\end{align}
		and there is $\pi^* \in \mathcal{BM}(\mu,\nu)$ such that
		\begin{align*}
			\mathcal{P}_c(\mu,\nu) = \int_{O}\int_{O} c(x,y)\pi^*(dx,dy).
		\end{align*}
 	\end{theorem}

 	Before giving the proof of Theorem \ref{thm:weakdual}, we start by noting the relationships between the various formulations of our problem.  Recall first the following correspondence between (possibly randomized) stopping times and subharmonic martingale measures. While randomized stopping times have many equivalent representations, it suffices to consider their law as probability measures on $C(\R^+,\R^d)\times \R^+$ that disintegrate with respect to Wiener measure as a measure $\eta$ on $\R^+$ such that $A_t = \int_0^t \eta(ds)$ is adapted to the filtration of the Brownian motion.  The only topology we consider is the weak* convergence of the probability law on $C(\R^+,\R^d)\times \R^+$.
	\begin{lemma}[See for example Theorem 2.1 in \cite{ghoussoub-kim-lim-radial2017}]\label{lem:measure-stoppingtime}
		Let $\sigma$ and $\rho$ be probability measures on $\overline{O}$ and let $(B_t)_t$ denote Brownian motion starting according to $B_0\sim \sigma$. Then, 
		\begin{enumerate}

		\item	For each (possibly randomized) stopping time $\tau\leq \tau_O$ satisfying $B_\tau\sim \rho$, there is $\pi\in \mathcal{BM}(\sigma,\rho)$ such that $(B_0,B_\tau)\sim \pi$.
		
		\item Conversely, for each $\pi\in\mathcal{BM}(\sigma,\rho)$, there exists a (randomized) stopping time $\tau\leq \tau_O$, such that $(B_0,B_\tau) \sim \pi$. 
		\end{enumerate}
	\end{lemma}

	We shall need the following characterization for superharmonic lower semicontinuous functions. We include a proof for completeness.
		\begin{lemma}\label{lem:superharmonic}
	For $\phi\in LSC(\overline{O})$ the following are equivalent:
	\begin{enumerate}
		\item $\Delta \phi(x)\leq 0$ in the sense of viscosity for all  $x\in O$;
		\item For any stopping time $\tau\leq \tau_O$ and all $y\in \overline{O}$, we have
		\begin{align}\label{eqn:superharmonic}
			\phi (y) \ge \mathbb{E}\big[ \phi (B^y_\tau) \big]. 
		\end{align}
	\end{enumerate}
	\end{lemma}
\begin{proof} 
	First we suppose that $\phi\in LSC(\overline{O})$ satisfies $\Delta \phi(x)\leq 0$ in the sense of viscosity for all $x\in O$.  There is then a sequence of smooth functions $\phi^i\in C^\infty(\overline{O})$ with $\phi^i \le \phi$ in $O$, and constants $\delta^i$ such that $\lim_{i\rightarrow \infty} \phi^i(x)=\phi(x)$ for all $x\in \overline{O}$, $\lim_{i\rightarrow \infty}\delta^i= 0$, and $\Delta \phi_i(x)\leq 0$ for $x\in O^{\delta^i}$, where $O^{\delta^i}$ is the set of points in $O$  with distance from $\partial O$ greater than $\delta^i$ (see Lemma \ref{lem:viscosity_approximation}, or similar results in \cite{fukuda1993uniqueness}, \cite{crandall1992user}, \cite{Sylvestre2015}.)

	For $y\in \partial O$, $B^y_\tau=B^y_{\tau_O}=y$ and (\ref{eqn:superharmonic}) holds.  For $y\in O$ and any stopping time $\tau\leq \tau_O$, we fix $\delta>0$ such that $y \in O^\delta$ and set $\tau^\delta = \tau\wedge \tau_{O^\delta}$. Then, by It\^{o}'s Lemma
	\begin{align*}
		\phi (y) =\lim_{i\rightarrow \infty} \phi^i(y) =&\ \lim_{i\rightarrow \infty}\mathbb{E}\Big[ \phi^i (B^y_{\tau^\delta}) -\int_0^{\tau^\delta} \Delta \phi^i(B_t^y)dt\Big]\\
		\ge&\ \lim_{i\rightarrow \infty}\mathbb{E}\big[ \phi^i (B^y_{\tau^\delta})\big]=\mathbb{E}\big[ \phi (B^y_{\tau^\delta}) \big].
	\end{align*} 
	Then taking $\delta\rightarrow 0$ we have
	$$
		\liminf_{\delta\rightarrow 0} \mathbb{E}\big[ \phi (B^y_{\tau^\delta}) \big] \geq \mathbb{E}\big[ \liminf_{\delta\rightarrow 0} \phi (B^y_{\tau^\delta}) \big]\geq \mathbb{E}\big[ \phi (B^y_\tau) \big],
	$$
	by lower semicontinuity of $\phi$, Fatou's lemma, and the continuity of Brownian paths.

	For the other direction (2) $\Rightarrow$ (1), we consider $w\in C^2(\overline{O})$ that touches $\phi$ from below at $y\in O$, i.e., $w(y)=\phi(y)$ and $w(x)\leq \phi(x)$ for all $x\in \overline{O}$.  Then, for all stopping times $\tau\leq \tau_O$, 
	\begin{align*}
		 \mathbb{E}\big[w(B_\tau^y)\big]-w(y)
		\le\ \mathbb{E}\big[\phi(B_\tau^y)\big]-\phi(y)\leq0,
	\end{align*} 
	which implies that $\Delta w(y)\leq 0$ for  $w \in C^2$.  This completes the proof. 
	\end{proof}
	\begin{proof}[\bf Proof of Theorem~\ref{thm:weakdual}]
		We sketch the proof of $\mathcal{P}_c(\mu, \nu) = \mathcal{D}_{c}'(\mu, \nu)$ as very similar results have appeared in \cite{beiglboeck2017optimal} and \cite{GKPS}. We assume $\psi$ and $J$ are continuous unless stated otherwise.\\
The constraint that $\pi\in \mathcal{BM}(\mu,\nu)$, defined in (\ref{eqn:BM}), is equivalent to:
		\begin{itemize}
			\item The measure $\pi$ on $O\times O$ is a finite and nonnegative;
			\item The second marginal of $\pi$ is $\nu$, or
				$$
					0=\sup_{\psi}\Big\{\int_{\R^d}\psi(y)\nu(dy)-\int_{\R^d}\int_{\R^d}\psi(y)\pi(dx,dy)\Big\},
				$$
				\item For each $x$, $\pi_x$ is in subharmonic order with $\delta_x$, and the first marginal is $\mu$, i.e.,
				\begin{align*}
					0=&\ \sup_{J; \Delta_y J\leq 0} \Big\{\int_{\R^d}\int_{\R^d}J(x,y)\pi(dx,dy)-\int_{\R^d}J(x,x)\mu(dx)\Big\}.
				\end{align*}
			\end{itemize}
Thus, we have
 		\begin{align*}
 			& \inf_{\pi\in \mathcal{BM}(\mu,\nu)} \int_{\R^d}\int_{\R^d} c(x,y) \pi(dx,dy)\\
 			&=\  \inf_{\pi\geq 0} \sup_{\psi}\sup_{J;\Delta_yJ\le 0}\Big\{ \int_{\R^d}\int_{\R^d} c(x,y) \pi(dx,dy) + \int_{\R^d} \psi(y) \nu(dy) \\
 			& \quad -\int_{\R^d}\int_{\R^d} \psi(y) \pi(dx,dy) +\int_{\R^d}\int_{\R^d} J(x,y)\pi(dx,dy) -\int_{\R^d} J(x,x)\mu(dx)\Big\},
 		\end{align*}
 		and the Fenchel-Rockafellar duality theorem allows us to interchange the infimum and supremum. This expression becomes
 		\begin{align*}
 			 \mathcal{P}_c(\mu, \nu)=&\ \sup_{\psi}\sup_{J;\Delta_yJ\le 0}\inf_{\pi\geq 0}\Big\{ \int_{\R^d}\int_{\R^d}\big(c(x,y)-\psi(y)+J(x,y)\big) \pi(dx,dy) \\
 		& \qquad 	+ \int_{\R^d} \psi(y) \nu(dy) -\int_{\R^d} J(x,x)\mu(dx)\Big\}\\
 			=&\ \mathcal{D}'_c(\mu,\nu).
 		\end{align*}
 		The attainment of a minimizer $\pi^*$ for $\mathcal{P}_c(\mu,\nu)$ is immediate from the compactness of probability measures in the weak* topology and the definition of $\mathcal{BM}(\mu,\nu)$, which makes it a weak* closed convex set in the space of Radon measures on $\overline{O}\times \overline{O}$.\\
To prove that 
 		$$
 			\mathcal{D}_{c}'(\mu, \nu) = \mathcal{D}_{c}(\mu, \nu):= \sup_{\psi\in LSC(\overline{O})}\Big\{\int_O \psi(y)\nu(dy)-\int_O J_\psi(x,x)\mu(dx)\Big\}, 
 		$$
 		we first show the inequality $\mathcal{D}_{c}'(\mu, \nu) \geq \mathcal{D}_{c}(\mu, \nu)$.  For that, let $\pi^*\in \mathcal{BM}(\mu,\nu)$ be where the infimum of $\mathcal{P}_c(\mu,\nu)$ is attained and its representation by a (randomized) stopping time $\tau^*$, cf.\ Lemma \ref{lem:measure-stoppingtime}. Then, for any $\psi \in LSC(\overline{O})$, 
 		\begin{align*}
 			\int_O \psi(y)\nu(dy)-\int_O J_\psi(x,x)\mu(dx)=&\ \mathbb{E}\big[\psi(B_{\tau^*})-J_\psi(B_0,B_0)\big]
 			\leq\ \mathbb{E}\big[\psi(B_{\tau^*})-J_\psi(B_0,B_{\tau^*})\big]\\
 			\leq&\ \mathbb{E}\big[c(B_{0},B_{\tau^*})\big]=\mathcal{P}_c(\mu, \nu) =\mathcal{D}_c'(\mu,\nu),
 		\end{align*}
 		where we have used the definition (\ref{eqn:J_psi}) of $J_\psi$. Hence $\mathcal{D}_{c}(\mu, \nu)\leq \mathcal{D}_{c}'(\mu, \nu)$. 

 		For the other direction, we consider $(\psi,J)\in \mathcal{A}_c$, then $J_\psi(x,y)\leq J(x,y)$ for $(x,y)\in O\times O$ since 
		by Lemma \ref{lem:superharmonic} and the definitions (\ref{eqn:A_c}) and (\ref{eqn:J_psi}), we have
 		\begin{align*}
 			J_\psi(x,y) =&\ \sup_{\tau\leq \tau_O}\mathbb{E}\big[\psi(B^y_\tau)-c(x,B^y_\tau)\big]
 			\leq\ \sup_{\tau\leq \tau_O}\mathbb{E}\big[J(x,B^y_\tau)\big]\leq J(x,y).
 		\end{align*}
 		It follows that 
 		\begin{align*}
 			\int_O \psi(y)\nu(dy)-\int_O J_\psi(x,x)\mu(dx)\geq&\ \int_O \psi(y)\nu(dy)-\int_O J(x,x)\mu(dx), 
 		\end{align*}
		and $D'_c(\mu,\nu)\leq D_c(\mu,\nu)$. This completes the proof. 
	\end{proof}

\section{A weakly compact set of $D$-superharmonic functions in Sobolev class}\label{sec:BD}

In this section, we introduce the following convex set of functions $\mathcal{B}_D$, which plays a key role in the sequel. These are the lower semicontinuous $D$-superharmonic functions that are non-positive and zero on the boundary of $O$.  A key property will be that these functions can be equivalently defined as members of the Sobolev class $H_0^1(O)$ that are non-positive and have their Laplacian bounded above by $D$ weakly.

\begin{definition}\label{def:normalized}  Let $O$ be a bounded convex open subset of $\mathbb{R}^d$.
 We say that a function $\psi  \in LSC(\overline{O})$  is in the set $\mathcal{B}_D$, if the following properties hold:
\begin{enumerate}
\item $\psi(y)=0$ for all $y\in \partial O$;
 \item $\psi(y)\leq 0$ for all $y\in O$;
 \item $\psi$ is {\em $D$-superharmonic}, in the sense that for all stopping times $\tau\leq \tau_O$,
 \begin{align*}
	\psi(y) \geq   \mathbb{E}\big[ \psi (B^y_\tau) -D\tau\big].
\end{align*}
(Equivalently, $\Delta \psi(y)\leq D$ for all $y\in O$ in the sense of viscosity by Lemma \ref{lem:superharmonic}.)
\end{enumerate}
\end{definition}
\noindent The functions in this class $\mathcal{B}_D$ have a uniform lower bound following from the maximum principle. Indeed,  let  $u_O$ be the solution to 
\begin{align}\label{eqn:u-O}
 \begin{cases}
  \Delta u_O(x)  = 1   & x\in O, \\
   u_O(x) =0   & x\in\partial O.
\end{cases}
\end{align}
Equivalently, $u_O(y)=-\mathbb{E}\big[\tau_O\big]$ where $\tau_O$ is the exit time for the Brownian motion beginning at $B_0^y=y$.
Since $\Delta \psi(y) \le D$,  we see 
 that  
\begin{align}\label{eqn:zero-bound}
&\psi(y) \ge \mathbb{E}\big[\psi(B^y_{\tau_O})-D\tau_O] =Du_O(y)\geq -D\sup_{z\in \overline{O}}|u_O(z)|,
\end{align}
which provides a lower bound depending only on $O$ and $D$.

We now prove a key property of  $\mathcal{B}_D$, namely that it can be identified with a closed convex bounded (hence weakly compact) subset of the Sobolev class $H_0^1(O)$, equipped with 
the norm 
$
	\|f\|_{H^1_0(O)}^2=\int_O |\nabla f(x)|^2dx.
$
\begin{proposition}\label{lem:BD_properties}
	A function $\psi$ is in $\mathcal{B}_D$ if and only if only if --up to a set of Lebesgue measure zero-- it 
	satisfies
	\begin{enumerate}
		\item $\psi\in H_0^1(O)$;
		\item $\psi(y) \leq 0$ for a.e.\ $y\in O$;
		\item $\Delta \psi(y)\leq D$ weakly, in the sense that
		$
			\int_O \big[\nabla \psi(y)\cdot \nabla \phi(x)+D\phi(x)\big]dx\leq 0 
		$
		for all $\phi \in H_0^1(O)$ such that $\phi\leq 0$ a.e. on $O$.
	\end{enumerate}
Furthermore, there is a constant $M$ dependent only on $D$ and $O$ such that 
	\begin{equation}\label{eqn:psi-infty-bound-long}
		\|\psi\|_{H_0^1(O)}\leq M \quad \hbox{for all $\psi\in \mathcal{B}_D$. }
	\end{equation}
In other words, 	 $\mathcal{B}_D$ can be identified with a closed bounded convex subset of $H_0^1(O)$. 
\end{proposition}
\begin{proof}
First, use Lemma~\ref{lem:approxBD} to fix a bounded  convex domain $\tilde{O}$ containing $O$ such that we can approximate $\psi\in \mathcal{B}_D$ by smooth functions on $\tilde{O}$ that satisfy:  $\lim_{i\rightarrow \infty} \psi^i(x)=\psi(x)$, $\psi^i(x)\leq 0$ and $\Delta\psi^i(x)\leq D$ for $x\in \tilde{O}$, and $\psi^i(y)=0$ for $y\in \partial \tilde{O}$.  Notice that $\psi^i(x)\geq  D u_{\tilde{O}}(x)$ for $x\in \tilde{O}$ for $u_{\tilde{O}}$ defined similarly as  \eqref{eqn:u-O}.

From the weak Laplacian bound with $\phi=\psi^i$, we get 
		\begin{align*}
			\int_{\tilde{O}} |\nabla \psi^i(x)|^2dx  
			 \le  -\int_{\tilde{O}} D \psi^i(x)dx
			 \le  D^2\|u_{\tilde{O}}\|_{L^1(O)}.
		\end{align*}
Hence, there is a subsequence of $\psi^i$ that converges weakly in $H^1(O)$ to an equivalence class of $\psi$ in $H_0^1(O)$ with $\|\psi\|_{H_0^1(O)}^2\leq D^2 \|u_{\tilde{O}}\|_{L^1(\tilde{O})}$. \\
The properties in \textit{(2)} and \textit{(3)} follow as they are stable under weak convergence in $H^1(O)$.

Conversely, if $\psi\in H_0^1(O)$ satisfies $\Delta \psi\leq D$ weakly and $\psi\leq 0$, we can easily check that the extension of the function $\psi-Du_O$ to $\R^d$ by zero is superharmonic in the sense that the average integral 
$$
	r \ \longmapsto \ \frac{1}{|B_r(y)|} \int_{B_r(y)} \big(\psi(z)-Du_O (z)\big)dz
$$
is monotonically decreasing in $r$. It follows  that $\psi$ has a representative $\tilde \psi$ that is lower semicontinuous and $D$-superharmonic in the sense of viscosity. See for instance the notes \cite{Sylvestre2015}.  This representative $\tilde \psi$ is everywhere nonpositive and $\tilde \psi(y)\geq Du_O(y)$ for all $y\in \overline{O}$, thus $\tilde \psi$ is zero on the boundary. We have shown the lower semicontinuous representative $\tilde \psi$ of $\psi$ lies in $\mathcal{B}_D$, completing the proof.
\end{proof}

We now define the superharmonic envelope of a function $h\in LSC(\overline{O})$ to be
\begin{align}\label{eqn:SH}
	h^{SH}(y)=\sup_{\tau\leq \tau_O} \mathbb{E}\big[h(B_\tau^y)\big].
\end{align}
We note that the definition of $J_\psi$ in (\ref{eqn:J_psi}) means that $y\mapsto J_\psi(x,y)$ is the superharmonic envelope of $y\mapsto \psi(y)-c(x,y)$.  
We will require in the sequel a few results on superharmonic envelopes.  

\begin{lemma}\label{lem:variational}
	Given  $\phi\in H_0^1(O)\cap C^2(\overline{O})$, then its superharmonic envelope $\phi^{SH}$ as defined in (\ref{eqn:SH}) is the unique minimizer of the variational problem, i.e.
	$$
		\phi^{SH}={\rm argmin}\Big\{\int_O |\nabla u(y)|^2dy;\ u\in H_0^1(O),\ u\geq \phi\Big\}.
	$$ 
\end{lemma}
\begin{proof}
	Take $\phi\in H_0^1(O)\cap C^2(\overline{O})$ and let $\tilde{\phi}\in H_0^1(O)$ be the unique minimizer of the variational problem; uniqueness follows from the strict convexity of $u \mapsto \int |\nabla u|^2$.  The optimality of $\tilde{\phi}$ implies that $\tilde{\phi}$ satisfies $\Delta \tilde{\phi}\leq 0$ weakly (see \cite{kinderlehrer1980introduction}).  As in the proof of Proposition \ref{lem:BD_properties}, we have that $\tilde{\phi}$ has a lower semicontinuous representative that satisfies $\Delta \tilde{\phi}(x)\leq 0$ for $x\in O$ in the sense of viscosity and $\tilde{\phi}(x)\geq \phi(x)$ for $x\in \overline{O}$.  This implies that $\phi^{SH}(y)\leq \tilde{\phi}(y)$ for all $y\in \overline{O}$ since from Lemma~\ref{lem:superharmonic},
	$$
		\phi^{SH}(y)=\sup_{\tau\leq \tau_O}\mathbb{E}\big[\phi(B_\tau^y)\big]\leq\sup_{\tau\leq \tau_O}\mathbb{E}\big[\tilde{\phi}(B_\tau^y)\big]\leq \tilde{\phi}(y).
	$$
For any smooth superharmonic function $\hat{\phi}\in H_0^1(O)\cap C^2(\overline{O})$ greater than $\phi$, we have  by the same argument as above, $\hat{\phi}\geq \phi^{SH}$. It follows that $\phi^{SH}-\hat{\phi}\in \mathcal{B}_{D'}$ for $D'= \sup_{x\in O} -\Delta \hat{\phi}(x)$.  Use now Proposition \ref{lem:BD_properties} to show that $\phi^{SH}\in H_0^1(O)$ and satisfies $\Delta \phi^{SH}\leq 0$ weakly. It follows that
	$$
		\|\phi^{SH}\|_{H_0^1(O)}^2 = \int_O  \phi^{SH}(y)\big(-\Delta \phi^{SH}(y)\big)dy\leq \int_O  \tilde{\phi}(y)\big(-\Delta \phi^{SH}(y)\big)dy \leq \|\tilde{\phi}\|_{H_0^1(O)}\|\phi^{SH}\|_{H_0^1(O)}.
	$$
	Thus $\phi^{SH}$ is a minimizer of the variational problem, 
	 hence from uniqueness we have $\phi^{SH}=\tilde{\phi}$.  
\end{proof}
We shall need the following property of the norm of $H^{-1}(O)$, which is dual to $H^{1}_0(O)$. 
\begin{lemma}\label{lem:dualnorm}
	Let  $\tau$ and $\sigma$ be two stopping times such that $\tau\leq \sigma\leq \tau_O$, and suppose that $B_\tau \sim \rho\in H^{-1}(O)$.  Then, the distribution $\gamma$ of $B_\sigma$ belongs to $H^{-1}(O)$ and satisfies
	$$
		\|\gamma\|_{H^{-1}(O)}\leq \|\rho\|_{H^{-1}(O)}.
	$$
\end{lemma}
\begin{proof}
	For each $\phi\in H_0^1(O) \cap C^2(\overline{O})$,  the superharmonic envelope $\phi^{SH}$ satisfies $\| \phi^{SH}\|_{H_0^1(O)} \le \|\phi\|_{H_0^1(O)}$ by Lemma \ref{lem:variational}.
	Therefore, 
	\begin{align*}
		\int_O \phi(y) \gamma(dy) \leq \int_O \phi^{SH}(y)\gamma(dy) \leq \int_O \phi^{SH}(y) \rho(dy) \leq \|\rho\|_{H^{-1}(O)}\|\phi\|_{H_0^1(O)}. 
	\end{align*}
	Here,  the second inequality is due to the subharmonic order $\rho \prec_{SH} \gamma$ and subharmonicity of $-\phi^{SH}$. This proves the lemma as smooth functions are dense in $H_0^1(O)$.
\end{proof}

The next lemma shows  continuity of functions in class $\mathcal{B}_D$ with respect to stopping times.

\begin{lemma} \label{lem:Sobolev_stopping_time_convergence}
	We consider a sequence of randomized stopping times $\xi_i\leq \tau_O$, that converge weakly (in law) to a randomized stopping time $\xi_\infty$.  Then if $h\in \mathcal{B}_D$, we have 
	\begin{equation}\label{eqn:limit}
		\lim_{i \rightarrow \infty}\mathbb{E}\big[h(B_{\xi_{i}})\big]= \mathbb{E}\big[h(B_{\xi_\infty})\big].
	\end{equation}
\end{lemma}
\begin{proof}
	We first note that 
	\begin{equation}\label{eqn:liminf}
		\liminf_{i \rightarrow \infty}\mathbb{E}\big[h(B_{\xi_{i}})\big]\geq \mathbb{E}\big[h(B_{\xi_\infty})\big]
	\end{equation}
	follows as a consequence of the Portmanteau theorem using the lower semicontinuity of $h$ composed with the continuous paths of Brownian motion. 

To prove the opposite inequality, 
	\begin{equation}\label{eqn:limsup}
		\limsup_{i \rightarrow \infty}\mathbb{E}\big[h(B_{\xi_{i}})\big]\leq \mathbb{E}\big[h(B_{\xi_\infty})\big],
	\end{equation}
	we use that $h\in H_0^1(O)$ to control larger times and that $h$ is $D$-superharmonic to control small times.  We fix $\epsilon>0$ and select $\delta>0$ such that $\delta\leq \frac{\epsilon}{4D}$ and 
	\begin{align*}
		\mathbb{E}\big[h(B_0)\big]\leq \mathbb{E}\big[h(B_{\xi_\infty\wedge \delta})\big]+\frac{\epsilon}{4}.	
	\end{align*}
	Note that the latter is possible by (\ref{eqn:liminf}) since $\lim_{\delta\rightarrow 0}\xi_\infty\wedge \delta=0$.\\
	We decompose the expectation into two pieces,
	\begin{equation}\label{eqn:two_pieces}
		\mathbb{E}\big[h(B_{\xi_{i}})\big] = \mathbb{E}\big[h(B_{\xi_{i}})-h(B_{\xi_{i}\wedge \delta})\big]+\mathbb{E}\big[h(B_{\xi_{i}\wedge \delta})\big].
	\end{equation}
	The second term on the righthand side of (\ref{eqn:two_pieces}) satisfies (because of $D$-superharmonicity of $h$) that 
	$$
		\mathbb{E}\big[h(B_{\xi_{i}\wedge \delta})\big] \leq \mathbb{E}\big[h(B_0)\big]+D\delta \leq \mathbb{E}\big[h(B_{\xi_\infty\wedge \delta})\big]+\frac{\epsilon}{2}.
	$$\\
	We define $\tau_\delta = \tau_O\wedge \delta$.  For the first term on the righthand side of (\ref{eqn:two_pieces}), because $\xi_i$ and $\xi_i\wedge \delta$ coincide on the set where $\xi_i\leq \tau_\delta$, we have
	$$
		 \mathbb{E}\big[h(B_{\xi_{i}})-h(B_{\xi_{i}\wedge \delta})\big]= \mathbb{E}\big[h(B_{\xi_{i}\vee \tau_\delta})-h(B_{\tau_\delta})\big].
	$$
	For $\rho_\delta \sim B_{\tau_\delta}$, the distribution is in $H^{-1}(O)$, by comparison with the heat kernel $g_\delta$ on $\R^d$ since (by restricting to nonnegative test functions $w$ and extending by zero outside of $O$)
	\begin{align*}
		\|\rho_\delta\|_{H^{-1}(O)} =&\  \sup_{w\geq 0,\|w\|_{H^1_0(O)}\leq 1}\mathbb{E}\big[ w(B_{\tau_\delta})\big]\\
		\leq&\ \sup_{w\geq 0, \|w\|_{H^1_0(O)}\leq 1}\mathbb{E}\big[ w(B_{\delta})\big]= \|g_\delta\|_{H^{-1}(O)}=C(\delta).
	\end{align*}
	We have that $\xi_i \vee \tau_\delta \geq \tau_\delta$, and thus, by Lemma \ref{lem:dualnorm}, for $\rho_i \sim B_{\xi_i\vee \tau_\delta}$ and $\rho_\infty\sim  B_{\xi_\infty\vee \tau_\delta}$,
	$$
		\|\rho_i\|_{H^{-1}(O)}\leq C(\delta),\ \ \ \|\rho_\infty\|_{H^{-1}(O)}\leq C(\delta).
	$$
	Then weak convergence of $\xi_i$ implies $\rho_i\rightharpoonup \rho_\infty$ in $H^{-1}(O)$, and there is $I$ such that for $i\geq I$
	$$
		\mathbb{E}\big[h(B_{\xi_{i}\vee \tau_\delta})\big]\leq \mathbb{E}\big[h(B_{\xi_{\infty}\vee \tau_\delta})\big]+\frac{\epsilon}{2}.
	$$
Putting everything together we have that
	\begin{align*}
		 \mathbb{E}\big[h(B_{\xi_{i}})-h(B_{\xi_{i}\wedge \delta})\big]=&\ \mathbb{E}\big[h(B_{\xi_{i}\vee \tau_\delta})-h(B_{\tau_\delta})\big]\\
		 \leq&\ \mathbb{E}\big[h(B_{\xi_{\infty}\vee \tau_\delta})-h(B_{\tau_\delta})\big]+\frac{\epsilon}{2}\\
		 =&\ \mathbb{E}\big[h(B_{\xi_{\infty}})-h(B_{\xi_\infty\wedge \delta})\big] +\frac{\epsilon}{2}.
	\end{align*}
	Thus we have shown that for $i\geq I$,
		$\mathbb{E}\big[h(B_{\xi_{i}})\big]\leq \mathbb{E}\big[h(B_{\xi_{\infty}})\big]+\epsilon,$
	which implies (\ref{eqn:limsup}) and completes the proof of (\ref{eqn:limit}).
\end{proof}

The above property of $\mathcal{B}_D$ enables us to prove the following lemma regarding verification of  hitting times. We also include the case for $C(\overline{O})$. 
\begin{lemma} \label{lem:SH_hitting_time}
	Assume that either $h\in C(\overline{O})$ or that $h\in \mathcal{B}_D$, and let $h^{SH}$ be its superharmonic envelope defined in (\ref{eqn:SH}). 
	For fixed $y\in \overline{O}$, we let 
	$$
		\eta=\inf\{t;\ h^{SH}(B^y_t)=h(B^y_t)\}.
	$$ 
	The stopping time $\eta$  attains the supremum in the definition of $h^{SH}$. It also satisfies
	\begin{enumerate}
		\item $h(B^y_{\eta})=h^{SH}(B^y_{\eta})$,  
		\item for any stopping time $\sigma\leq \tau_O$, we have
	$
		h^{SH}(y)=\mathbb{E}[h^{SH}(B^y_{\sigma \wedge\eta})].
	$
	\end{enumerate}
\end{lemma}
\begin{proof} 
{\bf Case $h \in C(\overline{O})$:}
	Let us first show that there exist stopping times $\eta^i\leq \tau_O$ such that $h(B^y_{\eta_i})=h^{SH}(B^y_{\eta^i})$ and $\eta^i\rightarrow \eta$ weakly.
	Define for each $\epsilon >0$ the set 
	$$
		H_\epsilon=\{(t,\omega); h^{SH}(B_t (\omega))-h(B_t(\omega))=0,\ t<\eta+\epsilon\},
	$$
where each $\omega$ is a point in the probability space $\Omega$. Note that $h^{SH}=h=0$ along $\partial O$, therefore, for almost every $\omega$, there exists $t$ with $(t,\omega)\in H_\epsilon$ by the definition of $\eta$, thus the projection of $H_\epsilon$ on $\Omega$ has full measure.  By the `section theorem' in \cite{Dellacherie-Meyer-78} 
	 there exists a stopping time $\eta_\epsilon$ such that $(\eta_\epsilon(\omega),\omega)\in H_\epsilon$ whenever $\eta_\epsilon(\omega)<\infty$, and $\mathbb{P}[\eta_\epsilon<\infty]\geq 1-\epsilon$. Given a sequence $\epsilon^i$ converging to zero, we define $\eta^i: =\eta_{\epsilon^i} \wedge \tau_O$ 
	 which has a subsequence converging weakly as desired.\\
Using the continuity of $h$ and $h^{SH}$ and the continuity of Brownian paths, we have 
	$$
		\mathbb{E}[h^{SH}(B^y_{\eta})]=\lim_{i\rightarrow \infty}\mathbb{E}[h^{SH}(B^y_{\eta^i})]=\lim_{i\rightarrow \infty}\mathbb{E}[h(B^y_{\eta^i})]=\mathbb{E}[h(B^y_{\eta})].
	$$
The supremum of definition (\ref{eqn:SH}) is attained at a randomized stopping time $\tau$ by compactness \cite{edgar1982compactness}.
	Optimality of $\tau$ implies that 
	$$
		h^{SH}(y)=\mathbb{E}[h(B^y_{\tau})]=\sup_{\tau\leq \sigma\leq \tau_O}\mathbb{E}\big[h(B^y_{\sigma})]=\mathbb{E}\big[h^{SH}(B^y_{\tau})\big].
	$$
	Since $h^{SH}(B_\tau^y)\geq h(B_\tau^y)$, we have $h(B^y_{\tau})=h^{SH}(B^y_\tau)$ almost surely so $\tau\geq \eta$. Furthermore, $h^{SH}$ is superharmonic so we have
	$$
		\mathbb{E}[h^{SH}(B^y_\tau)]\leq\mathbb{E}[h^{SH}(B^y_{\eta})],
	$$
	which implies $h^{SH}(y)=\mathbb{E}[h(B^y_{\eta})]$.  Thus $\eta$ also attains the supremum of (\ref{eqn:SH}), which again implies that
	$$
		h(B^y_{\eta}) = \sup_{\eta\leq \sigma\leq \tau_O}\mathbb{E}\big[h(B^y_{\sigma})\, |\, B_\eta^y]=h^{SH}(B_\eta^y),
	$$  
	proving  \textit{(1)}. \\
Finally, for any stopping time $\sigma\leq \tau_O$, we have from the superharmonic property
	$
		h^{SH}(y)\geq \mathbb{E}[h^{SH}(B^y_{\sigma \wedge \eta})]\geq \mathbb{E}[h^{SH}(B^y_{\eta})]=h^{SH}(y),
	$
	which implies \textit{(2)} and completes the proof in the case where $h$ is continuous.

\noindent {\bf Case $h \in \mathcal{B}_D$:} 	
	If now $h$ is in $\mathcal{B}_D$, the above proof is still valid thanks to Lemma \ref{lem:Sobolev_stopping_time_convergence}. This can be used to obtain an optimal stopping time $\tau$ that maximizes
	$$
		\sup_{\tau}\mathbb{E}\big[h(B_\tau^y)\big],
	$$
	since the expectation of $h$ is then a weakly continuous function of the stopping times. The same lemma can also be used to carry on the limit of $(\eta^i)_i$ to $\eta$ in the above proof. \end{proof}

\section{Dual attainment in Sobolev class}\label{sec:dual_attainment}

In this section we prove one of the main results of the paper, namely, the attainment of the supremum in the dual problem.  This has been an elusive problem, which has previously only been fully resolved for $d=1$ in \cite{BLO-dual2017}. Recall that we assume  ${\rm supp\,}\mu$ and ${\rm supp\,}\nu$ are contained in a bounded convex open set $O \subset \mathbb{R}^d$, and by Theorem \ref{thm:weakdual}, we have the dual problem in the form
\begin{align*}
 \mathcal{D}_{c}(\mu, \nu) = \sup_{\psi  \in LSC(\overline{O})} \Big\{  \int \psi (y)\nu(dy)  - \int J_\psi (x, x) \mu(dx) 
\Big\}. 
\end{align*}
We now state our main result on dual attainment. 
\begin{theorem}\label{thm:strong_dual}
	Assume that $c\in C(\overline{O}\times \overline{O})$ and $y \mapsto c(x, y)$ is subharmonic and $D$-superharmonic, i.e.,  $0\leq \Delta_y c (x, y) \leq D$ in the sense of viscosity,  and that $\mu\prec_{SH} \nu$ as well as  $\mu\in H^{-1} (O)$.   There exists then $\psi^*  \in  \mathcal{B}_D$ that attains the maximum value of the dual problem, i.e., 
 	\begin{align}\label{eqn:int-limit}
		\mathcal{D}_{c} (\mu, \nu)= \int_O \psi^* (y) \nu(dy) - \int_O J_{\psi^*} (x,x) \mu(dx).
	\end{align}
\end{theorem}
\begin{remark}\label{rmk:makingsubharmonic}
As long as $\Delta_y c(x,y)\geq -M$, we can form a problem with subharmonic cost $\tilde{c}(x,y)=c(x,y)+h(y)$, with $h$ solving $\Delta h=M$.  The solutions to these two problems are equivalent in the sense that if $\tilde{\psi}$ is an optimizer for the cost $\tilde{c}$, then $\tilde{\psi}-h$ is an optimizer for the cost $c$.

Furthermore, for costs of the form $c(x,y)=|x-y|^\alpha$ where $\alpha>0$ and $d\geq 2$, a version of Theorem 4.1 applies with the additional assumption that the support of $\mu$ and $\nu$ are disjoint.  The argument is given in Theorem \ref{thm:distance-unique} for the case of the distance function, i.e.\ $\alpha=1$. 

 The subharmonicity condition $\Delta_y c(x, y) \ge 0$ is, however, more essential than the $D$-superharmonic property.
 For example, we have a counterexample to dual attainment in \cite{beiglbock2013model} (see also \cite{ghoussoub2015structure})  for the cost $c(x,y) = -|x-y|$. 
 \end{remark}
Before we prove the Theorem \ref{thm:strong_dual}, we prove a few lemmata that will allow us to utilize the weak compactness of the set $\mathcal{B}_D$ in $H^1_0$ (Proposition \ref{lem:BD_properties}) for the proof of dual attainment.

\begin{lemma}\label{lem:forlimits}
	If $\mu$ is a probability measure on $O$ and $\mu \in H^{-1}(O)$, then the map  $\psi \in H^1_0(O)\mapsto  \int_O J_{\psi}  (x,x) \mu (dx) $ is convex and  lower semicontinuous on $\mathcal{B}_D$.  	
\end{lemma}
\begin{proof}
We first prove the technical result that we may interchange the supremum and the expectation to obtain
$$
	\int_O J_{\psi}  (x,x) \mu (dx) = \sup_{\tau\leq \tau_O}\mathbb{E}\big[ \psi(B_\tau)-c(B_0,B_\tau)\big]	
$$
for the Brownian motion where $B_0\sim \mu$. To see this equality, fix $\epsilon>0$ and consider a measurable selection $\tau_x$ of stopping times for $B_0^x=x$ such that
$$
	J_\psi(x,x)\leq \mathbb{E}\big[ \psi(B_{\tau_x}^x)-c(x,B_{\tau_x}^x)\big]+\epsilon.
$$
 Integrate with respect to $\mu$ to get a stopping time $\tau$ for $B_0\sim \mu$ such that
$$
	\int_O J_{\psi}  (x,x) \mu (dx) \leq \mathbb{E}\big[ \psi(B_{\tau})-c(B_0,B_{\tau})\big]+\epsilon.
$$
For the other direction, note that any stopping time $\tau$ disintegrates as $\tau_x$, for $B_0^x=x$, and
$$
	\mathbb{E}\big[ \psi(B_\tau)-c(B_0,B_\tau)\big] = \int_O \mathbb{E}\big[ \psi(B_{\tau_x}^x)-c(x,B_{\tau_x}^x)\big]\mu(dx)\leq \int_O J_\psi(x,x)\mu(dx).
$$

Now consider $B_0 \sim \mu \in H^{-1}(O)$. From Lemma \ref{lem:dualnorm},  if $\tau\leq \tau_O$ and $B_\tau \sim \rho$ then $\rho \in H^{-1}(O)$.  Therefore, for each $\tau\leq \tau_O$, $\psi \in H_0^1(O) \mapsto \mathbb{E} \left[ \psi(B_{\tau}) - c(B_0, B_{\tau}) \right] $ is linear and continuous. This shows the desired convexity and lower semicontinuity as the map is given by the supremum of continuous linear functionals. 
\end{proof}

The next step is to show that the maximization of the dual problem $ \mathcal{D}_{c}(\mu,\nu)$  can be restricted to the smaller set $\mathcal{B}_D$.
\begin{proposition}\label{prop:normalized_dual}
We assume that   $y\mapsto c(x, y)$ is subharmonic and $D$-superharmonic, i.e., $0\leq \Delta_yc(x,y)\leq D$ as well as $\mu\prec_{SH} \nu$.  Then, it holds that 
\begin{align*}
 \mathcal{D}_{c}(\mu,\nu) = \sup_{ \psi  \in \mathcal{B}_D} \Big\{  \int_O \psi (y)\nu(dy)  - \int_O J_{\psi}(x,x) \mu(dx) \Big\}.
\end{align*}
\end{proposition}

The proof of this proposition is done as two improvements to $\psi$.  To maximize the dual value we wish to choose $\psi$ as large as possible while maintaining $\psi(x)\leq J_\psi(x,y)+c(x,y)$ for every $x\in \overline{O}$, which motivates the following lemma.
\begin{lemma}\label{lem:psi_semi_superharmonic}
	Suppose 
	$y\mapsto h(x,y)$ is $D$-superharmonic for each $x\in O$. 
	\begin{enumerate}
	\item  The function $\bar\psi(y):=\inf_{z\in \overline{O}}h(z,y)$	 is $D$-superharmonic.

	\item For $h(x,y)=J_\psi(x,y)+c(x,y)$ and if $y\mapsto c(x,y)$ is $D$-superharmonic, then
\begin{itemize}
	\item $\bar\psi$ is $D$-superharmonic,
	\item $\bar\psi(y) \geq \psi(y)$ for all $y\in \overline{O}$, and
	\item $J_{\bar\psi}(x,y)=J_\psi(x,y)$ for all $(x,y)\in \overline{O}\times \overline{O}$.
\end{itemize}
\end{enumerate}
\end{lemma}
\begin{proof} 
For any $y\in \overline{O}$ and $\epsilon>0$, there is $z\in \overline{O}$ such that
	$\bar\psi(y) +\epsilon \ge h(z,y),$ hence 
	  for any stopping time $\tau\leq \tau_O$,
	$$
		\bar\psi(y) +\epsilon \ge h(z,y)\geq \mathbb{E}\big[h(z,B^y_\tau)-D\tau\big]\geq \mathbb{E}\big[\bar \psi(B_\tau^y)-D\tau\big],
	$$
	where we have used that $y\mapsto h(x,y)$ is $D$-superharmonic and $\bar\psi(y')\leq h(z,y')$ for all $y' \in \overline{O}$. The first item (1) follows by letting $\epsilon \to 0$.

	For the second part of the lemma, we have that $y\mapsto c(x,y)+J_\psi(x,y)$ is $D$-superharmonic since $c$ is $D$-superharmonic and $J_\psi$ is superharmonic. That $\bar\psi(y) \geq \psi(y)$ is obvious from the fact that $\psi(y)\leq J_\psi(x,y)+c(x,y)$.  
	For the last item in (2), observe that $J_\psi\leq J_{\bar\psi}$ as
	$$
		J_\psi(x,y)=\sup_{\tau\leq \tau_O}\mathbb{E}\big[\psi(B_\tau^y)-c(x,B_\tau^y)\big]\leq\sup_{\tau\leq \tau_O}\mathbb{E}\big[\bar\psi(B_\tau^y)-c(x,B_\tau^y)\big]=J_{\bar \psi}(x,y).
	$$
	By definition of $\bar\psi$ we have that $\bar\psi(y)\leq J_\psi(x,y)+c(x,y)$ thus we have that for all $\tau\leq \tau_O$,
	$$
		\mathbb{E}\big[\bar\psi(B_\tau^y)-c(x,B_\tau^y)\big]\leq \mathbb{E}\big[ J_\psi(x,B_\tau^y)\big]\leq J_\psi(x,y),
	$$
	and it follows that $J_\psi \geq J_{\bar\psi}$.
\end{proof}

The second improvement of $\psi$ is slightly more subtle. We make use of the assumption that $\mu\prec_{SH}\nu$ so that subtracting a superharmonic function of $y$ from both $\psi$ and $J_\psi$ will not decrease the dual value.  In general, subtracting a superharmonic function would violate the constraint that $y\mapsto J_\psi(x,y)$ is superharmonic, but subtracting the superharmonic envelope of $\psi$ from $J_\psi$ miraculously does not violate this constraint given that $y\mapsto c(x,y)$ is subharmonic.  

\begin{proposition}\label{prop:equal}
	Assume that $c\in C(\overline{O}\times \overline{O})$ and for all $x\in \overline{O}$, $y\mapsto c(x,y)$ is subharmonic, i.e., $\Delta_y c(x,y) \ge 0$. 
 	Then, for each $\psi \in C(\overline{O})$,  we have 
 	\begin{align*}
 		J_{\psi -\psi^{SH}} (x,y) = J_\psi (x, y) -\psi^{SH}(y), \hbox{ for all $(x, y) \in \overline{O}\times \overline{O}$}. 
 	\end{align*}
\end{proposition}
\begin{proof}
 We fix $(x, y)$ in $\overline{O}\times \overline{O}$. 
The proof that $J_{\psi -\psi^{SH}} (x,y) \ge J_\psi (x, y) -\psi^{SH}(y)$ is easy and does not need the assumption $\Delta_y c(x, y) \ge0$. Indeed, we let $\bar \tau$ be a (randomized) stopping time that attains the supremum for $J_\psi$, so that 
 \begin{align*}
	J_{\psi}(x,y) = \mathbb{E}\big[ \psi (B^y_{\bar \tau})  - c(x, B^y_{\bar \tau}) \big].
\end{align*}
Using $\mathbb{E}\left[ \psi^{SH}(B^y_{\bar \tau}) \right] \le \psi^{SH}(y)$, we have
\begin{align*}
 J_\psi (x, y) - \psi^{SH} (y) 
 \le \mathbb{E}\big[ \psi (B^y_{\bar\tau})   - c(x, B^y_{\bar \tau})  - \psi^{SH}(B^y_{\bar \tau}) \big]
 \le J_{\psi -\psi^{SH}}(x,y).
\end{align*}
The last inequality is due to the definition of $J_{\psi -\psi^{SH}}$.

 The reverse inequality is important for the proof of 
 Proposition~\ref{prop:normalized_dual} and requires  the assumption $\Delta_y c(x, y) \ge 0$. 
We let $\tau$ be a (randomized) stopping time attaining the supremum for $J_{\psi-\psi^{SH}}$, so that  
\begin{align*}
 J_{\psi-\psi^{SH}} (x, y) =  \mathbb{E}\big[ \psi (B^y_\tau) - \psi^{SH} (B^y_\tau) - c(x, B^y_\tau) \big].
\end{align*}
We consider the first hitting time:
$ \eta = \inf \{t;\ \psi (B^y_t ) =\psi^{SH} (B^y_t) \}.$ 
We can write
\begin{align*}
&  \mathbb{E}\left[ \psi (B^y_\tau) - \psi^{SH} (B^y_\tau) - c(x, B^y_\tau)\right]\\
&= \mathbb{E} \Big[ \psi(B^y_{\tau \wedge \eta}) -\psi^{SH}(B^y_{\tau\wedge\eta}) - c(x, B^y_{\tau\wedge\eta}) \Big]\\
&  \quad +   \mathbb{E} \Big[ \psi (B_\tau^y) - \psi(B^y_{\tau \wedge \eta})  - \psi^{SH} (B^y_\tau) +\psi^{SH}(B^y_{\tau\wedge\eta}) \Big]
\\
&  \quad - \mathbb{E}  \Big[c(x, B^y_\tau) -c(x, B^y_{\tau\wedge\eta}) \Big]\\
& = I + II + III. 
\end{align*}
For $I$, use item (2) of Lemma \ref{lem:SH_hitting_time} and the definition of $J_\psi$, to see
\begin{align*}
	I \le J_{\psi}(x,y) -\psi^{SH} (y).
\end{align*}
For $III$, by the assumption $\Delta_y c(x, y) \ge 0$, we see that
 $III \le  0.$
For the term $II$, notice that 
\begin{align*}
 \hbox{if $\tau =\tau\wedge \eta$ then $\psi (B^y_\tau) = \psi(B^y_{\tau\wedge\eta})$ and $\psi^{SH}(B^y_\tau) = \psi^{SH}(B^y_{\tau\wedge\eta})$}.
\end{align*}
Moreover, by item (1) of Lemma \ref{lem:SH_hitting_time}
\begin{align*}
 \hbox{if $\eta = \tau \wedge\eta$ then $\psi (B^y_{\tau\wedge\eta}) = \psi^{SH}(B^y_{\tau\wedge\eta})$}.
\end{align*}
Recall  that $\psi \le \psi^{SH}$ always. Therefore, we can conclude that 
$ II \le 0.$
All of these together imply  that
$   J_{\psi -\psi^{SH}} (x,y)
   \le J_\psi (x, y) -\psi^{SH}(y), $
as desired, completing the proof. 
\end{proof}

We now have the necessary ingredients to prove Proposition \ref{prop:normalized_dual}, which is the last component of the proof of Theorem \ref{thm:strong_dual}.
\begin{proof}[\bf Proof of Proposition \ref{prop:normalized_dual}]
The inequality 
\begin{align*}
 \mathcal{D}_{c}(\mu,\nu) \geq \sup_{ \psi  \in \mathcal{B}_D} \Big\{  \int_O \psi (y)\nu(dy)  - \int_O J_{\psi}(x,x) \mu(dx) \Big\}
\end{align*}
follows directly from the definitions of $\mathcal{D}_c$ and $\mathcal{B}_D$ as $\mathcal{B}_D\subset LSC(\overline{O})$.

For the reverse inequality,  first notice that for each $\phi \in LSC(\overline{O})$ there exists a sequence of continuous functions $\phi^i$ such that $\lim_{i\to \infty} \phi^i (x) = \phi(x)$, $\forall x \in \overline{O}$; for example, one can consider the inf-convolution as in the proof of Lemma~\ref{lem:viscosity_approximation}. Therefore,   it suffices to prove that for any  $\psi \in C(\overline{O})$ we can modify it to $\bar \psi \in \mathcal{B}_D$ such that 
\begin{align*}
 \int_O \psi (y) \nu(dy) - \int_O J_{\psi}(x, x) \mu(dx) \le  \int_O \bar \psi (y) \nu(dy) - \int_O J_{\bar \psi}(x, x)\mu(dx).
\end{align*}
(We can then apply the modification to a maximizing sequence of $\mathcal{D}_{c}$ to get another maximizing sequence but now from the class $\mathcal{B}_D$.)
 \subsubsection*{\bf Improvement 1 ($\psi=0$ on $\partial O$  and $\psi\leq 0$ in $O$)} We first modify $\psi$ to $\psi-\psi^{SH}$, 
using the superharmonic envelope $\psi^{SH}$ given in \eqref{eqn:SH}. 
We see that 
\begin{align*}
 & \int_O \psi (y) \nu (dy) - \int_O J_\psi(x,x) \mu(dx) \\
 & \le \int \left(\psi(y) -\psi^{SH} (y)\right) \nu (dy) - \int \left(J_{\psi} (x, x) -\psi^{SH}(x) \right)\mu(dx)
\end{align*}
because  $\psi^{SH}$ is super-harmonic and $\mu \prec_{SH}\nu$. 
It is important to notice that from  Proposition~\ref{prop:equal}, $J_{\psi-\psi^{SH}} = J_\psi -\psi^{SH}$. Also, notice that because $\overline{O}$ is compact, the continuous function $\psi$ is uniformly continuous, and so are  $\psi-\psi^{SH}$  and $J_{\psi-\psi^{SH}}$. 

\subsubsection*{\bf Improvement 2 ($\Delta \psi \leq D$ in $O$)} Now, modify the function $\psi -\psi^{SH}$ further, to 
$$
	\bar \psi(y)=\inf_{z\in \overline{O}} \big\{J_{\psi-\psi^{SH}}(z,y)+c(z,y)\},
$$ 
as in  Lemma~\ref{lem:psi_semi_superharmonic} with $h=J_{\psi-\psi^{SH}}+c$. Here $\bar\psi$ is continuous on $\overline{O}$ following from uniform continuity of $J_{\psi-\psi^{SH}}$ and $c$. From the Lemma we have that $\bar\psi$ is $D$-superharmonic, $\psi-\psi^{SH} \le \bar \psi$, and  $J_{\bar \psi} = J_{\psi-\psi^{SH}}$. Then, the last line in the above inequality is less than or equal to 
\begin{align*}
  \int_O \bar\psi (y) \nu (dy) - \int_O J_{\bar \psi}(x, x) \mu(dx).
\end{align*}
getting the desired inequality. 

We also have  that $\bar\psi(y)\leq 0$ for all $y$, since 
\begin{align*}
	J_{\psi-\psi^{SH}}(x,y)+c(x,y)= \sup_{\tau}\mathbb{E}\big[\psi(B_\tau^y)-\psi^{SH}(B_\tau^y)-c(x,B_\tau^y)+c(x,y)\big]\leq 0,
\end{align*}
and for $y\in \partial O$, $\bar\psi(y)=0$ since $J_{\psi-\psi^{SH}}(x,y)=-c(x,y)$. Thus, by definition this implies that $\bar\psi\in \mathcal{B}_D$. This completes the proof. 
\end{proof}
\begin{proof}[\bf Proof of Theorem \ref{thm:strong_dual}]
We use Proposition \ref{prop:normalized_dual} to find a sequence $\psi_i\in \mathcal{B}_D$ such that
$$
	\mathcal{D}_{c} (\mu, \nu)= \lim_{i\rightarrow \infty}\Big\{\int_O \psi_i (y) \nu(dy) - \int_O J_{\psi_i} (x,x) \mu(dx)\Big\}.
$$
Uniform boundedness of $\| \psi_i\|_{H^1_0 (O)}$ given by \eqref{eqn:psi-infty-bound-long} of Proposition \ref{lem:BD_properties} implies there is a weak limit $\psi^* \in H_0^1(O)$ of $\psi_i$. 
Note also that such a weak limit preserves the property $\psi^* \le 0$ as well as $\Delta \psi^* \le D$ in the weak sense, so by the equivalence of Proposition \ref{lem:BD_properties} we have $\psi^* \in \mathcal{B}_D$.  Since $\mu\in H^{-1}(O)$ and $\mu\prec _{SH}\nu$ we have $\nu\in H^{-1}(O)$ from Lemma \ref{lem:dualnorm}, and 
$$
	\displaystyle \lim_{i\to\infty} \int_O \psi_i(y) \nu(dy) = \int_O \psi^*(y) \nu(dy).
$$  
On the other hand,  from the lower semicontinuity shown in Lemma~\ref{lem:forlimits}, we have 
$$\displaystyle \lim_{i\to\infty} \int_O J_{\psi_i}(x,x) \mu(dx) \ge  \int_O J_{\psi^*}(x,x) \mu(dy), $$ which then implies 
\begin{align*}
	\int_O \psi^*(y) \nu(dy) -  \int_O J_{\psi^*}(x,x) \mu(dx)  \ge \mathcal{D}_{c} (\mu, \nu). 
\end{align*}
Since $\psi^*\in LSC(\overline{O})$, the above two inequalities are in fact equalities, showing the identity  \eqref{eqn:int-limit}. This completes the proof. 
\end{proof}

We now demonstrate the verification properties for how the dual optimizers pair with the primal minima. We will use these general results to prove uniqueness of the optimal stopping times in Section \ref{sec:hitting_time} after introducing the crucial assumption on the cost,  the stochastic twist condition in  Section \ref{sec:ST}.

	Let $\pi^*$ be an optimizer of $\mathcal{P}_c(\mu, \nu)$, and let $\tau^*$ be the corresponding optimal (randomized) stopping time. 
	Notice that for the disintegration $\pi^* (dx,dy) =\pi_x^* (dy)\mu(dx)$, the measure $\pi_x^*$ describes the distribution of the stopped Brownian paths $B^x_{\tau^*}$. 
	The dual optimizers $(\psi^*, J_{\psi^*})$ in Theorem~\ref{thm:strong_dual} are useful tools for us to characterize this measure. For example, we have the following important property. 

\begin{theorem}\label{thm:verification}
Under the same assumptions as in Theorem~\ref{thm:strong_dual}, 
 the optimal Brownian path stops at the contact set, namely, for $\pi^*$-a.e. $(x,y) \in O\times O$, 
\begin{align*}
 	J_{\psi^*} (x, y) = \psi^* (y) - c(x, y),
\end{align*}
and in particular $\tau^*\geq \eta:= \inf\{t  \ ; \ J_{\psi^*} (B_0, B_t) = \psi^* (B_t) - c(B_0, B_t)\}$.

Furthermore, if $\sigma(dx,dy)=\sigma_x(dy)\mu(dx)$ satisfies $0\prec_{SH}\sigma_x\prec_{SH}  \pi_x^*$ for each $\mu$-a.e. $x$,  then
\begin{align}\label{eqn:in-order}
	\int_{O\times O} J_{\psi^*}(x,y)\pi^*(dx,dy)=\int_{O\times O} J_{\psi^*}(x,y)\sigma(dx,dy)=\int_{O} J_{\psi^*}(x, x)\mu(dx).
\end{align}
If $\xi$ is a (nonrandomized) stopping time corresponding to $\sigma$, then $\xi\leq \tau^*$ and for $\mu$ a.e.\ $x$ and $\xi$ restricted to the paths with $B_0=x$,
\begin{align}\label{eqn:equalwithxi}
\mathbb{E}\big[ J_{\psi^*} (x, B_\xi^x) \big]  = J_{\psi^*}(x, x).
\end{align}
Also, for $\sigma$ a.e.\ $(x,y)$ and the stopping time $\tau^*-\xi$ restricted to paths satisfying $B_\xi=y$,
\begin{align}\label{eqn:equalwithxi2}
 \mathbb{E}\big[ J_{\psi^*} (x, B_{\tau^*-\xi}^y)\big]  & = 
  J_{\psi^*} (x, y).
\end{align}
In particular, these hold for $\xi=\eta$.
\end{theorem}

\begin{proof}

For $\sigma$ satisfying $\sigma(dx,dy)=\sigma_x(dy)\mu(dx)$ with $\delta_x\prec_{SH}\sigma_x\prec_{SH}\pi_x$ for each $\mu$-a.e. $x$, from the superharmonic property of $J_{\psi^*}$
	we have the inequalities
	$$
		\int_{O\times O} J_{\psi^*}(x,y){\pi}^*(dx,dy)\leq \int_{O\times O} J_{\psi^*}(x,y)\sigma(dx,dy)\leq\int_{O} J_{\psi^*}(x, x)\mu(dx).
	$$
Therefore, from  the obvious inequality  $J_{\psi^*}(x,y)\geq \psi^*(y)-c(x,y)$,  we see that 
	\begin{align*}
		\int_{O\times O} c(x,y){\pi^*}(dx,dy)&\ge  \int_{O\times O} \big[\psi^*(y)-J_{\psi^*}(x,y)\big]{\pi^*}(dx,dy)\\&\ge  \int_O \psi^*(y)\nu(dy)-\int_{O\times O} J_{\psi^*}(x, y) \sigma (dx, dy)\\
		&\ge  \int_O \psi^*(y)\nu(dy)-\int_O J_{\psi^*}(x, x)\mu(dx).
	\end{align*}
	Strong duality (Theorem~\ref{thm:strong_dual}) implies that the first and the last end of these inequalities are the same, making all the inequalities equalities. In particular, this  implies that $ J_{\psi^*} (x, y) = \psi^* (y) - c(x, y).$ for $\pi^*$-a.e. $(x,y)$. The equalities \eqref{eqn:in-order} also follow. Then the equalities \eqref{eqn:equalwithxi} and \eqref{eqn:equalwithxi2}  follow from  \eqref{eqn:in-order}  and the disintegration of $\pi^*$, $\sigma$, with respect to $\sigma$, $\mu$, respectively. In other words,
	$$
		\int_O \mathbb{E}\big[ J_{\psi^*} (x, B_\xi^x) \big]\mu(dx)=\int_O J_{\psi^*}(x,x)\mu(dx)
	$$
	and (\ref{eqn:equalwithxi}) follows since $J_{\psi^*} (x, B_\xi^x)\leq J_{\psi^*}(x,x)$, and
	$$
		\int_{O\times O} J_{\psi^*}(x,y){\pi}^*(dx,dy)=\int_{O\times O} \mathbb{E}\big[ J_{\psi^*} (x,B_{\tau^*-\xi}^y)\big] \sigma(dx,dy)=\int_{O\times O} \mathbb{E}\big[ J_{\psi^*} (x,y)\big] \sigma(dx,dy)
	$$
	and (\ref{eqn:equalwithxi2}) follows similarly.
\end{proof}

\section{A Stochastic Twist Condition}\label{sec:ST} 
We now start the discussion of our results on uniqueness and characterization of the optimal Brownian stopping time. In this section, 
	we give a stochastic version of the twist condition  \eqref{eqn:twist} on the cost in the deterministic optimal transport theory. This condition  will allow us to prove later that the optimal stopping time is uniquely given by the first hitting time to the contact set. 

	\begin{definition}\label{def:ST}
		We say that $c$ satisfies the stochastic twist ({\bf ST}) condition at $(x,y)$ if for each stopping time $\xi\leq \tau_O$,
\begin{align}\label{eqn:ST}
			\mathbb{E}\big[\nabla_xc(x,B^y_{\xi})\big] = \nabla_xc(x,y) \quad \Longrightarrow	\quad 	\xi=0.
\end{align}

	\end{definition} 
 Notice that this stochastic twist condition {\bf ST} is not a direct generalization of the usual twist condition in optimal transport theory.  In particular, the quadratic cost $c(x,y)=|x-y|^2$ does not satisfy {\bf ST}, because $\nabla_x |x-y|^2 = 2(x-y)$, therefore the equality in \eqref{eqn:ST} holds for any $\tau \ge 0$ due to the martingale property.

	We now collect some examples of costs that satisfy {\bf ST}:
	\begin{itemize}
		\item	Our most important example of a cost satisfying {\bf ST} is the distance function $c(x,y)=|x-y|$ for dimensions $d\geq 2$ and points $x\not=y$.  The gradient is $\nabla_xc(x,y)=\frac{x-y}{|x-y|}\in S^{d-1}$, valued in the unit sphere. In particular we have $\nabla_xc(x,y)\cdot \nabla_xc(x,y) =1$ and $\mathbb{E}\big[\nabla_xc(x,y)\cdot \nabla_xc(x,B_\sigma^y)\big]<1$ for any stopping time $\sigma>0$, which implies {\bf ST}. 
		Note that the same argument shows that any (differentiable) Riemannian distance function $d(x, y)$ satisfies {\bf ST}, since $\nabla_x d(x, y)$ is always a unit tangent vector at $x$.

		\item A more general class of costs with {\bf ST} can be described by the local condition: for each $x$,  there exist convex functions $f_x$ such that  
		\begin{align}\label{eqn:local_twist}
 			y\mapsto  f_x\big(\nabla_x c (x, y)\big) \ {\rm \ is\ strictly\ superharmonic}.
		\end{align}
		In this case, we have for $\sigma\not=0$
		\begin{align*}
			f_x\big(\mathbb{E}\big[\nabla_x c(x,B^y_\sigma)\big]\big)\leq&\ \mathbb{E}\big[f_x\big(\nabla_x c(x,B^y_\sigma)\big)\big]\\
			<&\ f_x\big(\nabla_x c(x,y)\big).
		\end{align*}

		\item A simple subclass of the previous example are the separable costs
		$$
			c(x,y)=g(x) h(y)
		$$
		that satisfy $\nabla g(x)\not= 0$ and $y\mapsto h(y)$ is either strictly superharmonic or strictly subharmonic.   In either case, $\nabla_x c(x,y)=\nabla g(x) h(y)$, and  we select $f_x(z)= \pm\nabla g(x)\cdot z$ where the sign is positive if $h$ is superharmonic and negative if $h$ is subharmonic.  

		\item For cost functions that only satisfy
		$$
			\Delta_y\nabla_xc(x,y)\not={\bf 0} \ {\rm for\ all}\ (x,y)\in O\times O,
		$$
		a localized version of {\bf ST} holds, i.e.\ there is $\delta>0$ such that if $0 \le \tau \le \delta$ and (\ref{eqn:ST}) then $\tau=0$. Note however, that  in dimension one, this condition is sufficient to imply {\bf ST} (i.e., that\ $\Delta_y\nabla_xc(x,y)=c_{yyx}(x,y)>0$, which appears in \cite{hobson2012model} and \cite{HLT2015}). 
	\end{itemize}  

\begin{remark}
 One can also consider a version of the stochastic twist condition for martingale transport:
 	\begin{definition}\label{def:ST-martingale}
		We say that $c$ satisfies the martingale twist ({\bf MT}) condition at $(x,y)$ 
if for each probability measure $\sigma$ such that $\delta_y \prec_C \sigma$ (for the convex order), we have
\begin{align}\label{eqn:ST-martingale}
	\int_O \nabla_xc(x,z) \sigma(dz)  = \nabla_xc(x,y)
		& \quad \Longrightarrow	\quad 	\sigma=\delta_y.
\end{align}
\end{definition} 
The above examples for {\bf (ST)} hold for {\bf (MT) } if we  replace the superharmonicity, subharmonicity with concavity, convexity, respectively. In particular, the distance cost $c(x, y) = |x-y|$ is {\bf (MT)}, and this property was crucially used in \cite{ghoussoub2015structure} to prove uniqueness and structure of optimal martingale transport under various conditions. 
\end{remark}

\section{Uniqueness of the Monge solution: optimal stopping as a hitting time}\label{sec:hitting_time}

The stochastic twist condition {\bf (ST)} allows us to prove uniqueness of the optimal Brownian martingale and characterize it as the first hitting time to a barrier set. 
 We will assume the technical assumption $\mu(\partial {\rm supp\,}\mu)=0$ and the simplifying structural assumptions that $\mu\wedge\nu =0$, which in most cases can be omitted by applying Lemma \ref{lem:stayput}.
   
\begin{theorem}\label{thm:hitting-unique}
Suppose, in addition to the assumptions of Theorem \ref{thm:strong_dual}, that $x\mapsto c(x,y)$ is $C^1$ and $c$ satisfies the stochastic twist condition ({\bf ST}) for all $(x,y)\in O\times O$.  
Assume further that $\mu \ll Leb$, $\mu(\partial {\rm supp\,}\mu)=0$,  and $\mu\wedge\nu =0$.
Then,  there exists a unique optimal stopping time that is given by
		\begin{align}\label{eqn:optimal_hitting}
			\eta:=\inf\{t;\ J_{\psi^*}(B_0,B_t)=\psi^*(B_t)-c(B_0,B_t) \}
		\end{align}
		where $\psi^*$ is the dual optimizer of Theorem \ref{thm:strong_dual}, and $J_{\psi^*}$ is the value function satisfying \eqref{eqn:J-psi}. 
\end{theorem}

We will need several technical lemmata that address differentiability issues for $J_{\psi^*}(x,y)$ in our proof of Theorem~\ref{thm:hitting-unique}.
	The dynamic programming principle for $J_\psi$ allows us easily verify the following remarkable Lipschitz continuity. 
	\begin{lemma}\label{lem:J_Lipschitz}
 		Assume that for each $y \in \overline{O}$, we have that $\| x\mapsto c(x, y)\|_{Lip} \le K$ for some constant $K>0$ (independent on $y$). 
		Then for each $y \in \overline{O}$, 
			$\| x\mapsto J_\psi (x, y)\|_{Lip} \le K.$
	\end{lemma}
	
\begin{proof}
 This is an easy conclusion from the definition of $J_\psi$ by stopping times, \eqref{eqn:J-psi}, because 
 $$
 \| x\mapsto \mathbb{E} \left[ \psi (B^y_\tau) - c(x, B^y_\tau)  \right]\|_{Lip}\le K,$$ for each $\tau$. 
The supremum over those Lipschitz functions is again Lipschitz, with the same Lipschitz constant $K$. 
\end{proof}

 The following two lemmas deal with the differentiability of two relevant integrals (expected values).  We first verify harmonicity of $y \mapsto J_{\psi^*} (x, y)$ in a small neighborhood. 
\begin{lemma}\label{lem:Jharmonic}
Use the same assumptions and notation as in Theorem~\ref{thm:hitting-unique}. 
Then, for each $x\in {\rm int}({\rm supp\,}\mu)$, 
 $y   \mapsto J_{\psi^*} (x, y)$ is harmonic in an open neighborhood around $x$.  
\end{lemma}
\begin{proof}
For $\mu$-a.e.\ $x \in  {\rm int}({\rm supp\,}\mu)$, choose an open ball $V_\epsilon(x)  \subset  {\rm int}({\rm supp\,}\mu)$ centered at $x$, with radius $\epsilon>0$ such that $\nu(V_\epsilon(x))=0$.  For each $r\le\epsilon$, let $\xi^r$ be the first hitting time to $\partial V_r(x)$. 
Define $u : V_\epsilon (x) \to \R$ as 
\begin{align*}
u (y) = \mathbb{E}\left[  J_{\psi^*} (x, B_{\xi^\epsilon}^y )\right] . 
\end{align*}
Because of Markov property of the Brownian motion, $u$ satisfies the mean value property, so, is harmonic. 
Recall that $y   \mapsto J_{\psi^*} (x, y)$ is superharmonic. Therefore, $J(x, y) \ge u(y)$ for all $y \in V_\epsilon(x)$. 
Moreover, because of our assumption 
 $\mu\wedge\nu =0$, 
we have $\xi^r \le \tau^*$, $0\le r\le \epsilon$, for the optimal stopping $\tau^*$  of $\mathcal{P}_c(\mu, \nu)$. Therefore,  from the verification theorem (Theorem~\ref{thm:verification}), we see that for $\mu$-a.e. $x$, 
\begin{align}\label{eqn:meanvalueJ}
 J_{\psi^*} (x, x) = \mathbb{E}\left[J_{\psi^*}(x, B_{\xi^r}^x) \right] = \mathbb{E}\left[J_{\psi^*}(x, B_{\tau^*}^x) \right]. 
\end{align}
In the case $r=\epsilon$, we have $J_{\psi^*}(x, x) = u(x)$. Then, for other $r$, \eqref{eqn:meanvalueJ} and the inequality $J_{\psi^*} \ge u$ imply that $J_{\psi^*} (x, y) = u (y)  $ for a.e. $y \in V_\epsilon(x)$. Now lower semi-continuity of $J_{\psi^*}$ and the inequality $J_{\psi^*} \ge u$ imply $J_{\psi^*} (x, y) = u (y) $ for all $ y \in V_\epsilon(x)$. 
To have the harmonicity for all $x \in  {\rm int}({\rm supp\,}\mu)$, not just $\mu$-a.e., use the (Lipschitz) continuity of $x \mapsto J_{\psi^*} (x, y)$, to extend  this harmonicity to  all $x \in {\rm int}({\rm supp\,}\mu)$. 
This completes the proof. 
\end{proof}
{
\begin{lemma}\label{lem:diffintegral}
Use the same assumptions and notation as in Theorem~\ref{thm:hitting-unique}. 
Let $\tau^*$ be an optimal stopping time of our problem $\mathcal{P}_c (\mu, \nu)$. 
Let $\zeta$ be any stopping time with  $\zeta \le \tau^*$
satisfying 
\begin{align}\label{eqn:expectation-identity}
  \mathbb{E}\left[ J_{\psi^*} (x, B^x_\zeta ) \right]
  = \mathbb{E}\left[ \psi^* ( B^x_\zeta) - c (x, B^x_\zeta)  \right] \quad \hbox{for $\mu$-a.e. $x$}. 
  \end{align}
  (In particular $\zeta=\eta$,  the hitting time defined in \eqref{eqn:optimal_hitting}, satisfies these  from Lemma~\ref{lem:SH_hitting_time} (1).) 
Then, for $\mu$-a.e. $x$
  the functions 
\begin{align*}
& h\mapsto J_{\psi^*} (x+h, x), 
 \quad  h\mapsto  \mathbb{E}\left[ J_{\psi^*} (x+h, B^x_{\zeta})  \right],   \quad \hbox{and}\quad h\mapsto  \mathbb{E}\left[ J_{\psi^*} (x+h, B^x_{\tau^*})  \right]  
\end{align*}
 are differentiable at $h=0$ and 
\begin{align}\label{eqn:equalderivatives}
 & \quad \frac{d}{dh}\Big|_{h=0} J_{\psi^*} (x+h, x)\\\nonumber 
&  = \frac{d}{dh}\Big|_{h=0}  \mathbb{E}\left[ J_{\psi^*} (x+h, B^x_{\zeta})  \right]=  \mathbb{E}\left[ - \nabla_x c (x, B^x_{\zeta})  \right] 
  \\\nonumber
& = \frac{d}{dh}\Big|_{h=0}  \mathbb{E}\left[ J_{\psi^*} (x+h, B^x_{\tau^*})  \right]  =
  \mathbb{E}\left[ - \nabla_x c (x, B^x_{\tau^*})  \right] . 
\end{align}
\end{lemma}
\begin{proof}
From  Lemma~\ref{lem:Jharmonic} and the gradient estimates of harmonic functions, the function $y \mapsto J_{\psi^*}(x, y)$ is locally Lipschitz in an open neighborhood of  each  $x \in {\rm int}({\rm supp\,}\mu)$.  
Here,  the local Lipschitz constant is uniform in $x$  in a neighborhood, since $c$ is continuous and the function $J_{\psi^*}$ is bounded in $O\times O$ (which follows from the boundedness of $\psi^* \in \mathcal{B}_D$.  
Combining this with  the fact  $\| x\mapsto J_{\psi^*}(x, y)\|_{Lip} \le C$ (from Lemma~\ref{lem:J_Lipschitz}) for each $y$ we get the function $(x, y)\mapsto 
J_{\psi^*}$ is locally Lipschitz on an open neighborhood $\mathcal{N}$ of the diagonal set $\{ (x, x) \ | \ x \in {\rm int}({\rm supp\,}\mu)\}$, contained in ${\rm int}({\rm supp\,}\mu)\times {\rm int}({\rm supp\,}\mu)$. 
  By Rademacher's theorem  $J_{\psi^*}$  is differentiable a.e. in the same set.  By Fubini's theorem and $\mu \ll Leb$, this implies that for  $\mu$-a.e. $x$,  the function $h \mapsto J_{\psi^*} (x+h, y)$ is differentiable at $h=0$ for a.e. $y$ in an open neighborhood of $x$. Because of the assumptions $\mu(\partial {\rm supp\,}\mu )=0$ and $\mu \ll Leb$, we can without loss of generality assume  that $x \in {\rm int\,} ({\rm supp\,}\mu)$, and the $\epsilon$-ball $V_\epsilon (x) \subset {\rm int} ({\rm supp\,}\mu)$ and  for all sufficiently small $h$,  the function $y \in V_\epsilon(x)  \mapsto J(x+h, y)$ is harmonic.

Choose  a stopping time $\xi$ such that  $\xi \le \sigma^\epsilon$ for the first hitting time $\sigma^\epsilon$ to the sphere $\partial V_\epsilon (x)$ and  
$B_\xi^x \sim \rho \ll Leb$. 
Then from the bound  $\| x\mapsto J_{\psi^*}(x, y)\|_{Lip} \le C$ and the dominated convergence theorem, we see that 
\begin{align*}
\frac{d}{dh}\Big|_{h=0}  \mathbb{E}\left[ J_{\psi^*} (x+h, B_\xi^x)\right] = \mathbb{E}\left[  \frac{d}{dh}\Big|_{h=0} J_{\psi^*} (x+h,B_\xi^x )\right]. 
\end{align*}
In particular, the derivative  exists. 
We now use the harmonicity of $y \in V_\epsilon(x) \mapsto J_{\psi^*} (x+h, y)$ for sufficiently small $h$ to see 
 \begin{align*}
  J_{\psi^*} (x+h, x)= \mathbb{E}\left[ J_{\psi^*} (x+h, B^x_\xi)  \right].
  \end{align*}
So, $h\mapsto J_{\psi^*}(x+h, x)$ is differentiable at $h=0$.

 As $\mu \wedge \nu =0$ and $V_\epsilon (x) \in {\rm int}({\rm supp\,}\mu)$, we see that $\sigma^\epsilon  \le \tau^*$, thus $\xi \le \tau^*$. 
 This order implies 
\begin{align*}
 J_{\psi^*}(x+h, x) = \mathbb{E}\left[ J_{\psi^*} (x+h, B^x_\xi)  \right] 
 & \ge   \mathbb{E}\left[ J_{\psi^*} (x+h, B^x_{\tau^*})  \right]  \\
 & \ge   \mathbb{E}\left[ \psi^* ( B^x_{\tau^*}) - c (x+h, B^x_{\tau^*})  \right] .
 \end{align*}
 Note that from Theorem~\ref{thm:verification} (in particular, \eqref{eqn:equalwithxi}) we have 
 $$ J_{\psi^*} (x, x) = \mathbb{E}\left[ J_{\psi^*} (x, B^x_\xi)  \right]=    \mathbb{E}\left[ \psi^* ( B^x_{\tau^*}) - c (x, B^x_{\tau^*})  \right] .$$
  Therefore, at $h=0$ the above inequalities become equalities. 
Since  $J_{\psi^*}(x+h, x) = \mathbb{E}\big[ J_{\psi^*} (x+h, B^x_\xi)  \big]$ and $\mathbb{E}\left[ \psi^* ( B^x_{\tau^*}) - c (x+h, B^x_{\tau^*})  \right] $ are both differentiable at $h=0$, 
we see that  the function  
\begin{align*}
\hbox{$h\mapsto  \mathbb{E}\left[ J_{\psi^*} (x+h, B^x_{\tau^*})  \right]$  is differentiable at $h=0$,}
\end{align*} and  all these have the same derivatives:
\begin{align*}
&\frac{d}{dh}\Big|_{h=0} J_{\psi^*} (x+h, x) \\&= 
\frac{d}{dh}\Big|_{h=0} \mathbb{E}\left[ J_{\psi^*} (x+h, B^x_{\tau^*})  \right] = \frac{d}{dh}\Big|_{h=0}\mathbb{E}\left[ \psi^* ( B^x_{\tau^*}) - c (x+h, B^x_{\tau^*})  \right] \\&=
\frac{d}{dh}\Big|_{h=0}\mathbb{E}\left[- c (x+h, B^x_{\tau^*})  \right] =\mathbb{E}\left[- \nabla_x c (x, B^x_{\tau^*})  \right]  .
\end{align*}

For $\zeta$ it is not clear whether $\xi \le \zeta$. Therefore, to examine the derivative of the function $  \mathbb{E}\left[ J_{\psi^*} (x+h, B^x_\zeta)  \right]$ we notice that 
\begin{align*}
  J_{\psi^*} (x+h, x) \ge  \mathbb{E}\left[ J_{\psi^*} (x+h, B^x_\zeta)  \right]    \ge \mathbb{E}\left[ J_{\psi^*} (x+h, B^x_{\tau^*} ) \right]
\end{align*}
with equality at $h=0$ (because of  \eqref{eqn:equalwithxi}). This shows that  $h \mapsto \mathbb{E}\left[ J_{\psi^*} (x+h, B^x_\zeta)  \right]$ is differentiable at $h=0$  
and
\begin{align*}
 \frac{d}{dh}\Big|_{h=0}\mathbb{E}\left[ J_{\psi^*} (x+h, B^x_\zeta)  \right]  =\frac{d}{dh}\Big|_{h=0}\mathbb{E}\left[ J_{\psi^*} (x+h, B^x_{\tau^*})  \right] . 
\end{align*}
Moreover, 
\begin{align*}
  \mathbb{E}\left[ J_{\psi^*} (x+h, B^x_\zeta ) \right]
  \ge \mathbb{E}\left[ \psi^* ( B^x_\zeta) - c (x+h, B^x_\zeta)  \right], 
  \end{align*}
  and from the assumption on $\zeta$ 
  we have equality at $h=0$. This verifies 
\begin{align*}
\frac{d}{dh}\Big|_{h=0}\mathbb{E}\left[ J_{\psi^*} (x+h, B^x_\zeta)  \right]  &=   \frac{d}{dh}\Big|_{h=0}\mathbb{E}\left[ \psi^* ( B^x_\zeta) - c (x+h, B^x_\zeta)  \right] \\
& = \frac{d}{dh}\Big|_{h=0}\mathbb{E}\left[ - c (x+h, B^x_\zeta)  \right]=  \mathbb{E}\left[ - \nabla_x c (x, B^x_\zeta)  \right].
\end{align*}
All together these complete the proof. 
\end{proof}

}

We now use the above lemmata to give the proof of Theorem \ref{thm:hitting-unique}.
\begin{proof}[\bf Proof of Theorem~\ref{thm:hitting-unique}]
Let $\pi$ denote the probability measure on $\overline{O}\times \overline{O}$ corresponding to $B_\eta$, that is, $B_\eta^x \sim \pi_x$ for $\mu$-a.e. $x$, where $\pi_x$ is the disintegration $\pi (dx,dy)=\pi_x (dy)\mu(dx).$ Fix a pair  $(x,y)$ chosen $\pi$-a.e., in particular, for $x$ to satisfy the results of Lemma~\ref{lem:diffintegral} with $\zeta=\eta$.
Then,  consider a small ball $V_\epsilon (y)$ of radius $\epsilon >0$ around $y$.  
Define a stopping time $\zeta_\epsilon$ as 
\begin{align*}
\zeta_\epsilon = \begin{cases}
    \eta  & \text{ if $B_\eta^x \in V_\epsilon (y)$}, \\
    \tau^*  & \text{otherwise}.
\end{cases} 
\end{align*}
 Notice that $\zeta_\epsilon$ satisfies $\eta \le \zeta_\epsilon \le \tau^*$ and \eqref{eqn:expectation-identity} (e.g. from Lemma ~\ref{lem:SH_hitting_time}(1)) so that we can apply Lemma~\ref{lem:diffintegral}. Then, \eqref{eqn:equalderivatives} gives 
\begin{align*}
  \mathbb{E}\left[  \nabla_x c (x, B^x_{\zeta_\epsilon})  \right] 
  =
 \mathbb{E}\left[  \nabla_x c (x, B^x_{\tau^*})  \right] . 
\end{align*}
From this we see 
\begin{align*}
 0 &=\frac{1}{\mathbb{P} [  B^x_\eta \in V_\epsilon (y) ] }  \Big[\mathbb{E}\left[  \nabla_x c (x, B^x_{\tau^*})  \right] - \mathbb{E}\left[  \nabla_x c (x, B^x_{\zeta_\epsilon}) \right] \Big] \\
 & =  \mathbb{E}\big[  \nabla_x c (x, B^x_{\tau^*}) \ | \ B^x_\eta \in V_\epsilon (y)   \big] -  \mathbb{E} \big[\nabla_x c (x, B^x_\eta)  \ | \  B^x_\eta \in V_\epsilon (y) \big].
 \end{align*}
Letting $\epsilon \to 0$, we see that for $\xi=\tau^*-\eta$ restricted to the paths where $y=B^x_\eta$,
\begin{align*}
 \mathbb{E}\big[   \nabla_x  c(x, B^y_{\xi}) \big] -    \nabla_x c(x, y) =0\ \hbox{for $\pi$-a.e.\ $(x,y)$}.
\end{align*}
 We apply the stochastic twist condition {\bf (ST)} in Definition~\ref{def:ST}, and we get $\xi=0$. Since this holds for $\pi$-a.e. $(x, y)$, this implies $\tau^* = \eta$, completing the proof.  
\end{proof}

\section{The case of the distance function}\label{sec:distance} 
	We now consider the distance  cost $c(x, y) =|x-y|$. We focus on the multi-dimensional case $d\ge 2$, because for the $1$-dimensional case ($d=1$) our problem is equivalent to the martingale optimal transport and the uniqueness and structure of the optimal stopping is well known  \cite{beiglboeck-juillet2016}, \cite{hobson2011skorokhod}, \cite{hobson2012model}, \cite{hobson2015robust}. 
 We first get the following theorem as a corollary of Theorem~\ref{thm:hitting-unique}, where 	
 we assume the strict  separation assumption ${\rm supp\,}\mu \cap {\rm supp\,} \nu=\emptyset$ 
  to ignore the singularity at $x=y$.  
	Then  a localization argument allows us to remove this disjointness of supports in Theorem~\ref{cor:distance-unique} where  we show that there is a unique optimal randomized stopping time $\tau^*$ given by the hitting time of a barrier whenever $\tau^*>0$, assuming that $\mu$ and $\nu$ have densities $f\in C(\overline{O})$ and $g\in C(\overline{O})$,  respectively. 

\begin{theorem}\label{thm:distance-unique}
	 Use the same assumptions as in Theorem~\ref{thm:hitting-unique} except that $c(x,y)=|x-y|$.
		Assume further that  $d\geq 2$ and ${\rm supp\,}\mu \cap {\rm supp\,} \nu=\emptyset$. 
				Then, the following holds:
		
\begin{enumerate}
 \item  There exists a constant $D$ and $ \psi^*\in \mathcal{B}_D$ such that $(\psi^*, J_{\psi^*})$ maximize the dual problem. 
\item 
		There is a unique optimal  stopping time that is given by
		\begin{align}\label{eqn:distance_hitting_time-1}
			\eta = \inf\{t; J_{\psi^*}(B_0,B_t)=\psi^*(B_t)-|B_0-B_t|  \}. 
		\end{align}

\end{enumerate}
			\end{theorem}
	\begin{proof}
Let  $\tau^*$ be an optimal stopping time for the cost $c(x, y) = |x-y|$. 		
We let $\epsilon>0$ be such that $|x-y|\ge \epsilon$ for all $x\in {\rm supp}(\mu)$ and $y\in {\rm supp}(\nu)$.  Then we consider a smooth subharmonic function  $c^\epsilon(x,y) \le |x-y|$ such that $c^\epsilon(x,y)=|x-y|$ whenever $|x-y|\ge \epsilon$ . This can be easily constructed since for $d\ge 2$, $\Delta |x-y|>0$ whenever $|x-y|\not=0$.
Let $D$ be the constant with $0 \le \Delta_y c^\epsilon (x, y) \le D$.  
		
		First, observe that by construction of $c^\epsilon$ (and the separation of ${\rm supp\,}\mu$ and ${\rm supp\,}\nu$ by $\epsilon$),  $\tau^*$ is also an optimal stopping time for the cost $c^\epsilon$, and $$\mathbb{E}\left[c^\epsilon( B_0, B_\tau ) \right] = \mathbb{E}\left[|B_0 - B_\tau | \right].$$
		
		We now consider the dual optimizers for the cost $c^\epsilon$, namely,  $(\psi^*, J^\epsilon_{\psi^*})$ with $\psi^* \in \mathcal{B}_D$ obtained from Theorem \ref{thm:strong_dual}. 
	Here, $J^\epsilon_{\psi^*}$ is the value function with respect to $c^\epsilon$, that is, 
\begin{align*}
 J^\epsilon_{\psi^*} (x, y) := \sup_{\sigma} \mathbb{E}\big[ \psi^*(B^y_\sigma) - c^\epsilon (x, B^y_\sigma) \big].
\end{align*}
Then, from Theorem~\ref{thm:hitting-unique}, $\tau^* =\eta^\epsilon$, for $\eta^\epsilon$ given in \eqref{eqn:optimal_hitting}  with respect to $c^\epsilon$ and $J^\epsilon_{\psi^*}$ ; this also proves uniqueness of $\tau^*$. 

On the other hand, notice that because $c^\epsilon (x, y) \le |x-y|$, 
\begin{align}\label{eqn:JepsilonbiggerJ}
 J_{\psi^*} (x, y) := \sup_{\sigma} \mathbb{E}\big[ \psi^*(B^y_\sigma) - |x- B^y_\sigma| \big] \le J^\epsilon_{\psi^*} (x, y). 
 \end{align}
Therefore, 
\begin{align*}
 \int_O \psi^* (y)\nu (dy) - \int_O J_{\psi^*}(x,x) \mu(dx) &\ge
 \int_O \psi^* (y)\nu (dy) - \int_O J^\epsilon_{\psi^*} (x,x) \mu(dx)\\ & = \mathbb{E}\big[c^\epsilon( B_0, B_{\tau^*} ) \big]= \mathbb{E}\big[|B_0 - B_{\tau^*} | \big]. 
\end{align*}
This proves that the pair $(\psi^*, J_{\psi^*})$ is a dual optimizer for the cost $|x-y|$, and the above inequality is in fact an equality, thus,  applying  \eqref{eqn:JepsilonbiggerJ} we get 
\begin{align}\label{eqn:JJepsilonequal}
 J_{\psi^*} (x, y) = J^\epsilon_{\psi^*} (x, y) \hbox{ for $\pi^*$-a.e. $(x,y)$}
\end{align}
where $\pi^*$ is the optimal subharmonic martingale measure corresponding to $\tau^*$. For $\mu$-a.e.\ $x$, $y\mapsto J^\epsilon_{\psi^*}(x,y)$ is harmonic for $y$ satisfying $|y-x|\leq \epsilon$ by Lemma \ref{lem:Jharmonic}, so (\ref{eqn:JJepsilonequal}) and superharmonicity of $J_{\psi^*}$, imply that $J_{\psi^*}(x,y)\geq J^\epsilon_{\psi^*}(x,y)$ for all $y$ satisfying $|y-x|\leq \epsilon$. Then we see that $\tau^*=\eta^\epsilon=\eta$ satisfies   
\eqref{eqn:distance_hitting_time-1}. This completes the proof. 
\end{proof}

\begin{theorem}\label{cor:distance-unique}
	For $c(x,y)=|y-x|$ and $d\geq 2$, if $\mu\prec_{SH}\nu$, and $\mu$ and $\nu$ have densities $f\in C(\overline{O})$ and $g\in C(\overline{O})$,  then there is a unique optimal stopping time $\tau^*$ that is randomized only at time $0$.  The optimal stopping time is given by $\tau^*=0$ with  density $g\wedge f$
	and otherwise $\tau^*$ is the hitting time $\eta$,
	$$
		\eta = \inf\{t>0;\ (B_0,B_t)\in R\}
	$$
	for some $R\subset \overline{O}\times \overline{O}$ measurable.
\end{theorem}

We will first show that  the overlapping mass, if any, of the probability measures $\mu$ and $\nu$ 
stays put under any optimal solution $\tau^*$.   This was already shown in \cite{ghoussoub-kim-lim-radial2017} by using the monotonicity principle of \cite{beiglboeck2017optimal}.  We shall give here a  direct proof without using that principle. 
\begin{lemma}[See \cite{ghoussoub-kim-lim-radial2017}]\label{lem:stayput}
We use the assumptions and notation of Theorem~\ref{cor:distance-unique}, except we suppose that the cost function $c$ satisfies  following conditions: 
\begin{itemize}
 \item $c$  is continuous and $c(x,x)=0, \forall x$;
 \item  $c$ satisfies the triangle inequality: $c(x, y) \le c(x, z) + c(z, y)$, $\forall x, y, z$,
 while equality holds only when $y$ is on the unique geodesic (line segment in $\R^d$) 
 connecting $x$ and $z$.
\end{itemize}
  Then,  any optimal randomized stopping time $\tau^*$ stops at time $0$ with density $f\wedge g$, i.e.\ 
  \begin{align*}
  	\mathbb{E}\big[\mathbf{1}\{\tau^*=0\}\big]=\int_O f(x)\wedge g(x)\ dx.
  \end{align*}
 \end{lemma} 

\begin{proof}
Given an optimal randomized stopping time $\tau^*$, we let $h$ be the density it stops at time $0$, i.e.\
$$
	\mathbb{E}\big[\mathbf{1}\{\tau^*=0\}\big]=\int_O h(x)\ dx
$$
Notice that $h \le f \wedge g$. We will prove the equality. 
In the following we  use the convention in definitions that the value of a quotient is $1$ if the denominator vanishes.  

First, define a randomized stopping time $\sigma$ from the initial  distribution $\mu$ so  that $\sigma$  follows the stopping rule of $\tau^*$ with probability density $f - f \wedge g$ (of course, it is possible to stop at time $0$ for $\tau^*$ with a certain density) and stops at time $0$ otherwise. Namely, 
\begin{itemize}
\item the initial distribution of $\sigma$ is  $B_0\sim \mu$;
 \item $\sigma$ is randomized at the initial point $B_0$ as  
   \begin{align*}
\sigma = \begin{cases}
    \tau^*  & \text{with probability $1- \frac{f(B_0) \wedge g(B_0)}{f(B_0)}$}, \\
 0 & \text{with probability $\frac{f(B_0) \wedge g(B_0)}{ f(B_0)}$},
\end{cases} 
\end{align*}
\end{itemize}
In particular, for any continuous function $\phi$, we have 
\begin{align}\label{eqn:sigmaexp}
 \mathbb{E}\left[\phi(B_\sigma)\right] = \int \mathbb{E}\left[ \phi (B^x_{\tau^*}) \right](f(x)- f(x)\wedge g(x)) dx + \int \phi (x) f(x)\wedge g(x) dx.
\end{align}
 
Let $\hat \mu$ be the final  distribution of $\sigma$, i.e.\ $B_\sigma \sim \hat \mu$.  As $B_{\tau^*}\sim \nu$ has density, by the construction of $\sigma$ the distribution $\hat \mu$ also has density, say $\hat f$.  Observe that $\hat f \ge f \wedge g$.
We now define another randomized stopping time $\xi$ from the initial distribution $\hat \mu$ so that it follows the stopping rule of $\tau^*$ with probability density $f \wedge g$ and otherwise it stops at time $0$. More precisely,  
\begin{itemize}
 \item the initial distribution of $\xi$ is  $B_0 \sim \hat \mu$;
 \item $\xi$ is randomized at the initial point $B_0$ as 
    \begin{align*}
\xi = \begin{cases}
    \tau^*  & \text{with probability $\frac{f(B_0) \wedge g(B_0)}{\hat f(B_0)}$}, \\
 0 & \text{with probability $1-\frac{f(B_0) \wedge g(B_0)}{\hat f(B_0)}$}.
\end{cases} 
\end{align*}
\end{itemize}
We verify that $B_\xi$ has the same final distribution $\nu$ as  $B_{\tau^*}$. 
To see this  note that 
\begin{align}\label{eqn:xiexp}
 \mathbb{E}\big[\phi(B_\xi)\big] & = \int \mathbb{E}\big[ \phi (B^x_{\tau^*}) \big]f(x)\wedge g(x) dx + \int \phi (x) \big(\hat f(x) - f(x)\wedge g(x)\big) dx.
 \end{align}
 Here, realizing $\int \phi(x) \hat f(x)dx = \mathbb{E}\left[ \phi (B_\sigma) \right] $ and using \eqref{eqn:sigmaexp} we see that this equation implies
 $$\mathbb{E}\left[ \phi  (B _\xi)\right] =   \int \mathbb{E}\left[ \phi (B^x_{\tau^*})\right] f(x) dx  =  \mathbb{E}\left[ \phi (B_{\tau^*})\right],$$ verifying the claim $B_\xi\sim B_{\tau^*} \sim \nu$.

We now let $\tau$ be the randomized stopping time from the initial distribution $\mu$ that follows first $\sigma$ then $\xi$. That is, $\tau$ is the randomized stopping time on the random paths obtained by concatenating the random paths following the stopping rule of $\sigma$ with  the random paths following that of $\xi$. We have $\sigma \le \tau$ and that the final distribution of $\tau$ is the same as the final distribution of $\xi$.
Therefore,  $B_{\tau} \sim \nu$.

Observe using the fact $c(x,x) =0$ and \eqref{eqn:sigmaexp} that 
\begin{align*}
  \mathbb{E}\big[ c(B_0, B_{\sigma}) \big] =  \int \mathbb{E}\big[ c (x, B^x_{\tau^*}) \big]\big(f(x)- f(x)\wedge g(x)\big) dx .
  \end{align*}
Similarly, from the construction of $\tau$ and \eqref{eqn:xiexp}, 
  \begin{align*}
  \mathbb{E}\left[ c(B_\sigma, B_\tau) \right]& =  \mathbb{E}\left[ c(B_0, B_\xi) \right] = \int \mathbb{E}\left[ c(x,  B^x_{\tau^*}) \right]f(x)\wedge g(x) dx  .
   \end{align*}
This shows that 
\begin{align}\label{eqn:equalsigmatau}
  \mathbb{E}\big[ c(B_0, B_{\sigma}) \big] + \mathbb{E}\big[ c(B_\sigma, B_{\tau} )\big] 
 = \mathbb{E}\big[ c(B_0, B_{\tau^*}) \big].
\end{align}
On the other hand, from the optimality of $\tau^*$ (and the fact that $\tau^*$ and $\tau$ have the same initial and final distributions) and the triangle inequality of $c$, 
 we have 
\begin{align*}
 \mathbb{E}\big[ c(B_0, B_{\tau^*} )\big]  \le   \mathbb{E}\big[ c(B_0, B_{\tau} )\big]  \le  \mathbb{E}\big[ c(B_0, B_{\sigma}) \big] + \mathbb{E}\big[ c(B_\sigma, B_{\tau} )\big].
\end{align*}
Here, all these inequalities become equalities due to \eqref{eqn:equalsigmatau}.
In particular, we have the equality case of the triangle inequality where $B_\sigma$ is on the geodesic connecting $B_0$ and $B_{\tau}$, which holds only if 
\begin{align}\label{eqn:or}
 \hbox{$\sigma =0$ or $\sigma =\tau$ almost surely.}
 \end{align}
 Otherwise, we would find that where $\sigma>0$, $B_\tau = B_\sigma+\lambda (B_\sigma-B_0)$ for $\lambda>0$, and hence
 $$
 \mathbb{E}\big[(B_\tau-B_\sigma)\cdot (B_\sigma-B_0)\big|\mathcal{F}_\sigma\big]>0,
 $$
 which cannot hold do to the martingale property of Brownian motion.

We will analyze (\ref{eqn:or}) to draw our conclusion.  
For the random paths with $\sigma>0$, \eqref{eqn:or} and the definition of $\sigma$ implies  $\sigma =\tau=\tau^*$ so the point $B_\sigma$ lands at where $\xi=0$. Note that $\xi =0$ with probability 
\begin{align*}
  \frac{f(B_{\tau^*}) \wedge g(B_{\tau^*})}{\hat f(B_{\tau^*})}   \frac{h(B_{\tau^*})}{f(B_{\tau^*})}  +  \left(1- \frac{f(B_{\tau^*}) \wedge g(B_{\tau^*})}{\hat f(B_{\tau^*})} \right) =1
\end{align*}
This implies that if $\tau^*>0$ then $h(B_{\tau^*}) = f(B_{\tau^*})$. Therefore, whenever $f(x)>h(x)$ it must be the case that $h(x)=g(x)$ because $B_{\tau^*}=x$ only when $\tau^*=0$. This then implies that $h(x) = f(x) \wedge g(x)$, completing the proof. 
\end{proof}

\begin{proof}[\bf Proof of Theorem~\ref{cor:distance-unique}]
	We first reduce the problem to the case where the measures $\mu$ and $\nu$ have disjoint supports, then apply Theorem \ref{thm:distance-unique}. 
Indeed, from Lemma~\ref{lem:stayput}, 
we have that the optimal stopping time $\tau^*$ stops at time $0$ with density $g\wedge f$.  Thus we only need to characterize $\tau^*$ when $\tau^* >0$. 

We now show that if $\tau^*>0$, then $\tau^*$ is given by $\eta$, which is the hitting time of a barrier.  First, on the subset of the probability space where $\tau^*>0$, $\tau^*$ is optimal for transporting the mass $\mu^+$ with density $(f-g)^+$ to $\nu^+$ with density $(g-f)^+$.  These measures satisfy $\mu^+\wedge \nu^+=0$ and $\mu^+(\partial {\rm supp\,}\mu^+)=0$. 
 
We require one more step to reduce our setting to the case where the measures have strictly disjoint support so that we can  apply Theorem \ref{thm:distance-unique}. We introduce $\mu^+_\epsilon$ as the restriction of $\mu^+$ to points where the distance to $\partial {\rm supp\,}\mu^+$ is greater than $\epsilon$.  Then we let $\nu^+_\epsilon$ be the stopping distribution of the restriction of $\tau^*$ to the initial distribution $\mu^+_\epsilon$.  Notice that $\tau^*$ is still optimal to this restriction. Applying Theorem \ref{thm:distance-unique}, we have that this restricted problem has a unique optimal hitting time (equal to $\tau^*$) given by (\ref{eqn:distance_hitting_time-1}).  This defines the set $R$ whenever $|y-x|\geq \epsilon$.  Taking $\epsilon$ to zero we get that $\tau^*=\eta$ whenever $\tau^*>0$.

To prove that $\tau^*$ is unique, we suppose that $\tau_1$ and $\tau_2$ are both optimal randomized stopping times.  By the argument above we have that both stop at $t=0$ with density $g\wedge f$, and for $t>0$ are given by the hitting times $\eta_1$ and $\eta_2$.  We form the randomized stopping time $\tau'$ that stops at $\tau_1$ with probability $\frac{1}{2}$ and at $\tau_2$ with probability $\frac{1}{2}$, i.e.\ with Brownian martingale measure $\pi'=\frac{1}{2}\pi_1+\frac{1}{2}\pi_2$.  The same argument applies to show that if $\tau'>0$ then $\tau'=\eta'$, which is the hitting time to a barrier.  It follows that $\tau_1=\tau_2$ because otherwise there is a finite probability of finding $(B_0,B_t)\in R_1$ but $(B_0,B_t)\not\in R_2$ (or vice versa), which contradicts that $\tau'$ is the hitting time of a barrier.
\end{proof}
\begin{remark}\label{rml:f,g}
Notice that in Theorem~\ref{cor:distance-unique}, the condition $f, g \in C(\overline{O})$ can be relaxed. For example, it suffices that $g$ belongs to $H^{-1}(O)$ outside the support of $\mu$. 
\end{remark}
\appendix
\section{Approximations}

We give a couple of approximation (via inf-convolution) results discussed in \cite{lions1983optimal}, \cite{fukuda1993uniqueness}, \cite{crandall1992user}, which are used in Lemma \ref{lem:superharmonic} and Proposition \ref{lem:BD_properties}.
For $h\in LSC(\overline{O})$ we define the inf-convolution as
$$
	h_\epsilon(y)=\inf_{z\in \overline{O}} \big\{ h(z)+\frac{1}{2\epsilon} |y-z|^2\big\}.
$$

\begin{lemma}\label{lem:viscosity_approximation} 
	Given $h\in LSC(\overline{O})$, we have that $h_\epsilon$ is Lipschitz and semiconcave. If $\Delta h\leq D$ on $O$ in the sense of viscosity and ${\rm dist}(x,\partial O)^2>4\epsilon\|h\|_\infty$, then $\Delta h_\epsilon(x)\leq D$ in the sense of viscosity. Furthermore, $h_\epsilon$ has a distributional Hessian $\nabla^2 h_\epsilon$, that is a matrix-valued measure and satisfies $\Delta h_\epsilon \leq D$ in the weak sense for the open set $\{ y  \in O | \ {\rm dist}(y,\partial O)^2>4\epsilon\|h\|_\infty \}$.\\
In particular, there is a sequence of functions $h_i\in C^\infty(\overline{O})$ such that $h_i(x)\leq h(x)$ for $x\in O$ and for each $x\in O$ and $\delta>0$, there exists $I$ such that for $i\geq I$ we have $\Delta h_i(x)\leq D$ and $h(x)-h_i(x)\leq \delta$.
\end{lemma}
\begin{proof}
	For each $z$, $y\mapsto h(z)+\frac{1}{2\epsilon} |y-z|^2$ is Lipschitz with constant $\epsilon^{-1}{\rm diam O}$ and if we subtract $\frac{1}{2\epsilon}|y|^2$, it becomes concave.  These properties are inherited by the infimum.\\
Suppose now that $\Delta h(y)\leq D$ in the sense of viscosity for all $y\in O$, and ${\rm dist}(x,\partial O)^2> 4\epsilon \|h\|_\infty$.  Then there is $z\in O$ such that
	$
		h_\epsilon(x)=h(z)+\frac{1}{2\epsilon} |x-z|^2.
	$ 
Now suppose that $w$ touches $h_\epsilon(x)$ from below at $x$.  Then we define
	$$
		\phi(y):=w(x+y-z)-\frac{1}{2\epsilon}|x-z|^2.
	$$
	Then we have that
	$$
		\phi(z)= w(x)-\frac{1}{2\epsilon}|x-z|^2=h(z),
	$$
	and 
	$$
		\phi(y)\leq h_\epsilon(x+y-z)-\frac{1}{2\epsilon}|x-z|^2\leq h(y).
	$$
	Thus using the viscosity property $\Delta h \le D$ of $h$, we have $\Delta\phi(z)\leq D$, and since $\Delta\phi(z)=\Delta w(x)$, this shows that $\Delta h_\epsilon(x)\leq D$ in the sense of viscosity.\\
It follows from semi-concavity, that the distributional Hessian of $h_\epsilon$ is a matrix-valued measure.  Furthermore, it decomposes as $\nabla^2 h_\epsilon = M+T$, where $M\leq 0$,  $T\in L^\infty$ and $M \perp T$. 
  If there was a point with density at $x$ where ${\rm tr}(T)>D$, we could construct a $C^2$ function near $x$ satisfying $D<\Delta w(x)<{\rm tr}(T)(x)$ such that $w(x)=h_\epsilon(x)$ and $w(y)\leq h_\epsilon(y)$.  Thus for $x$ satisfying ${\rm dist}(x, \partial O)^2>4\epsilon\|h\|_\infty$ we have $\Delta h\leq D$ in the weak sense.\\
Finally, for the smooth approximation, we can define $\hat{h}_\epsilon(x)$ by extending $h_\epsilon$ outside of the domain $O$ and convolving with a smooth mollifier whose support shrinks as  $\delta\to 0$.  Note that $\Delta \hat h_\epsilon \le D$.  We can then subtract a small constant so that $\hat{h}_\epsilon \leq h_\epsilon$ but is converging uniformly as $\delta \to 0$ (from Lipschitz property of $h_\epsilon$).   The result on smooth approximation follows by choosing both $\epsilon$ and $\delta$ sufficiently small, where we use the fact that $h_\epsilon \le h$ and from the lower semicontinuity of $h$, $\liminf_{\epsilon\to 0} h_\epsilon \ge h$ for each $x$. This completes the proof. 
\end{proof}
For functions in $\mathcal{B}_D$, the boundary condition gives a cleaner smooth approximation as follows. 
\begin{lemma}\label{lem:approxBD}
	There is an open bounded convex  set $\tilde{O}$ (depending only on $O$ and $D$) such that $O \subset \tilde O$ and for any $0<\epsilon\leq 1$ and $h\in \mathcal{B}_D$, $h_\epsilon$ satisfies $h_\epsilon(x)\leq 0$ and $\Delta h_\epsilon(x)\leq D$ for $x\in \tilde{O}$ and $h_\epsilon(y)=0$ for $y\in \partial \tilde{O}$.  It follows that each $h\in \mathcal{B}_D$ can be approximated pointwise by smooth functions satisfying the properties above in some open bounded convex set containing $O$.
\end{lemma}
\begin{proof}
For each $h \in \mathcal{B}_D$, we extend  it to $\R^d$ by giving zero value on $\R^d \setminus O$. Notice that this extended function $h$ (using the same notation) still satisfies $\Delta h \le D$ in $\R^d$ in the sense of viscosity.  Define $h_\epsilon(y)=\inf_{z\in \R^d} h(z)+\frac{1}{2\epsilon}|y-z|^2$. As $h$ is bounded, we can apply the same reasoning as in the proof of Lemma~\ref{lem:viscosity_approximation} to see that  $\Delta h_\epsilon(x)\leq D$ for all $x\in \R^d$ in the sense of viscosity. Also it clearly holds $h_\epsilon \le h \le 0$. \\
Consider a uniform bound $M$ for functions on $B_D$, i.e.\ $\sup_{x\in O}\left|u_O(x)\right|$  as in \eqref{eqn:zero-bound}.
From this, observe that for $y$ satisfying ${\rm dist}(y, O)^2\geq 4\epsilon M$, we have $h_\epsilon(y)=0$.  
A further convolution by a smooth bump function allows us to approximate $h_\epsilon$ uniformly with smooth functions $h_\epsilon^i$ that also satisfy $\Delta h_\epsilon^i(x)\leq D$ and $h_\epsilon^i \le 0$, with support in a fixed convex bounded open domain  $\tilde O$ (independent of individual $h$), which contains $O$ and $|y-z|^2\geq 5 M$ for all $y\in \partial \tilde{O}$ and $z\in O$. By also letting $\epsilon \to 0$, point-wise convergence of $h_\epsilon^i$ to $h$ follows.  This completes the proof. 
\end{proof}

\section{A proof of the monotonicity principle} 
As a simple application of our dual attainment (Theorem ~\ref{thm:strong_dual} and  ~\ref{thm:verification}), 
we now provide an alternative  proof of the following version of the monotonicity principle of Beiglb\"ock, Cox, and Huesmann \cite{beiglboeck2017optimal} adapted to our setting:

\begin{theorem}[See  \cite{beiglboeck2017optimal}]\label{thm:monotonicity} With the same notation as in Theorem ~\ref{thm:strong_dual} and  ~\ref{thm:verification}, in particular we consider  
the disintegrated probability measure $\sigma(dx,dy)=\sigma_x(dy)\mu(dx)$ such that $0\prec_{SH}\sigma_x\prec_{SH}  \pi_x^*$ for each $\mu$-a.e. $x$, 
and its corresponding randomized stopping time $\xi$ (note that $\xi \le \tau^*$).

For $\pi^*$-a.e.  $(x, y)$ and $\sigma$-a.e. $(x', y')$, it holds that  if $y=y'$ then the stopping time $\tau^*-\xi$ restricted to paths with $B_\xi=y$ satisfies
 \begin{align*}
  c(x, y )  + \mathbb{E}\big[c(x', B^{y}_{\tau^*-\xi} )\big] \le  c(x', y) 
 +  \mathbb{E}\big[c(x, B^{y}_{\tau*-\xi}) \big]. 
  \end{align*}
\end{theorem}

\begin{proof}
Notice that from Theorem~\ref{thm:verification}, we have for $\pi^*$-a.e. $(x, y)$ and $\sigma$-a.e. $(x', y')$, 
\begin{align} 
 J_{\psi^*}(x, y) & = \psi^* (y) - c(x, y), \label{yes1}\\   
 J_{\psi^*}(x', B^{y'}_{\tau^*-\xi}) = J_{\psi^*}(x', B^{x'}_{\tau^*})& = \psi^* (B^{x'}_{\tau^*}) -  c(x', B^{x'}_{\tau^*})=\psi^* (B^{y'}_{\tau^*-\xi}) -  c(x', B^{y'}_{\tau^*-\xi}). \label{yes2}
\end{align}
From the dynamic programming formulation of $J_{\psi^*}$, we have  
\begin{align}\label{OK1}
J_{\psi^*}(x, y') 
  \ge \mathbb{E}\Big[  \psi^* (B^{y'}_{\tau^*-\xi}) - c(x, B^{y'}_{\tau^*-\xi}) \Big].
\end{align}
It also follows from (\ref{eqn:equalwithxi2}) that 
\begin{align}\label{OK2}
\mathbb{E}\Big[ J (x', B^{y'}_{\tau^*-\xi} ) \Big]  = J_{\psi^*}(x', y') &\ge \psi^* (y') - c(x', y'), 
\end{align}
where the first equality  follows from \eqref{eqn:equalwithxi} and the inequality from the dynamic programming principle. 
Taking the expectation in \eqref{yes2}, we see that under the condition  $y=y'$, the left hand sides of (\ref{yes1}), (\ref{yes2})
are equal to the left hand sides of the inequalities (\ref{OK1}) and (\ref{OK2}). Now, subtract the sum of the two equations from the sum of the two inequalities and cancel the terms  $\psi^*(y)$ and $\mathbb{E}\big[  \psi^*(B^{y}_{\tau^*-\xi})\big]$, to obtain 
\begin{align*}
  0 \ge  - c(x', y) -  \mathbb{E}\Big[  c(x, B^{y'}_{\tau^*-\xi})\Big]
  + c(x, y) + \mathbb{E}\Big[  c(x', B^{y'}_{\tau^*-\xi})  \Big], 
\end{align*}
hence completing the proof. 
 \end{proof}
\begin{remark}\label{rmk:monotonicity}
 The condition $y=y'$ together with  $\sigma \le \tau^*$
 gives a version of the stop-go pair of \cite{beiglboeck2017optimal}. 
 Our proof is similar in spirit to that of \cite{Guo-Tan-Touzi-2016monotone} where weak duality is used, while we use the strong duality (dual attainment) under the additional assumption  $0\le \Delta_y c (x, y) \le D$ among others.  Because of this last condition, our monotonicity result does not completely replace that of \cite{beiglboeck2017optimal} for the distance cost $|x-y|$. 
\end{remark}

\bibliography{Skorokhod}
\bibliographystyle{plain}

\end{document}